\documentclass{amsart}

\usepackage{amssymb,amsmath,graphicx}

\newtoks\prt

\numberwithin{equation}{section}

\newtheorem{thm}{Theorem}[section]

\newtheorem{question}[thm]{Question}

\newtheorem{fact}[thm]{Fact}
\newtheorem{lemma}[thm]{Lemma}

\newtheorem{cor}[thm]{Corollary}

\newtheorem{example}[thm]{Example}

\theoremstyle{definition}

\newtheorem{remark}[thm]{Remark}

\def\eqn#1$$#2$${\begin{equation}\label#1#2\end{equation}}

\def\fra{\mathfrak{A}}

\def\A{\mathcal A}

\def\C{\mathcal C}

\def\F{\mathcal F}

\def\ms{\mathcal S}

\def\M{\mathcal M}

\def\ce{\mathbb C}

\def\lin{Lindel\"of}

\def\co{\operatorname{co}}
\def\aco{\operatorname{aco}}
\def\ep{\varepsilon}

\def\en{\mathbb N}
\def\ef{\mathbb F}

\def\er{\mathbb R}
\def\qe{\mathbb Q}

\def\r{|}

\def\ov{\overline}

\def \ext {\operatorname{ext}}

\def \reg {\partial _{\kern1pt\text{reg}}}

\def\hom{\operatorname{hom}}
\def\ahom{\operatorname{ahom}}

\def\odd{\operatorname{odd}}

\def\dd{\operatorname{d}}

\def\di{\,\mbox{\rm d}}

\newcommand{\norm}[1]{\left\|#1\right\|}

\newcommand{\abs}[1]{\left| #1  \right|}

\renewcommand{\Re}{\operatorname{Re}}

\renewcommand{\Im}{\operatorname{Im}}

\newcommand{\setsep}{;\,}
\newcommand{\fr}{Fr\'echet\ }

\begin{document}

\title{Baire classes of affine vector-valued functions}

\author{Ond\v{r}ej F.K. Kalenda and Ji\v r\'\i\ Spurn\'y}

\address{ Charles University in Prague\\
Faculty of Mathematics and Physics\\Department of Mathematical Analysis \\
Sokolovsk\'{a} 83, 186 \ 75\\Praha 8, Czech Republic}
\email{kalenda@karlin.mff.cuni.cz}
\email{spurny@karlin.mff.cuni.cz}



\thanks{Our investigation was supported by the Research grant GA\v{C}R P201/12/0290. The second author was also
supported by The Foundation of Karel Jane\v{c}ek for Science and Research.}

\begin{abstract}
We investigate Baire classes of strongly affine mappings with values in Fr\'echet spaces.  We show, in particular, that the validity of the vector-valued Mokobodzki's result on affine functions of the first Baire class is related to the approximation property of the range space. We further extend several results known for scalar functions on Choquet simplices or on dual balls of $L_1$-preduals to the vector-valued case. This concerns, in particular, affine classes of strongly affine Baire mappings, the abstract Dirichlet problem and the weak Dirichlet problem for Baire mappings. Some of these results have weaker conclusions than their scalar versions. We also establish an affine version of the Jayne-Rogers selection theorem.
\end{abstract}

\maketitle

%%%%%%%%%%%%%%%%%%%%%%%%%%%%%%%%%%%%%%%%%%%%%%%%%%%

\section{Introduction}

Investigation of Baire hierarchy of affine functions on compact convex set forms an important
part of the Choquet theory. The abstract Dirichlet problem for Baire functions on Choquet simplices
was studied for example in \cite{Jel,spurny-cejm,spurny-aus,spu-ka,spu-aus}, descriptive properties of affine functions for example in \cite{lusp,lusp2,lusp-complex,spu-zel}. Large number of results valid for simplices were proved in a more general framework of $L_1$-preduals.
Recall that a Banach space $X$ is called an \emph{$L_1$-predual} (or a \emph{Lindenstrauss space}) if its dual $X^*$ is isometric to a space $L^1(X,\ms,\mu)$ for a measure space $(X,\ms,\mu)$. There are two types of $L_1$-preduals -- the real and complex ones. Unlike in most areas, there is a difference between the real and complex theories. More precisely, there are some results known in the real case and unknown in the complex case and, moreover, proofs in the complex setting are often more involved. However, our results hold in both cases in the same form.

Real $L_1$-preduals were in depth investigated in
papers \cite{effros-real,fakhoury1,fakhoury2,bednar1972concerning,Lau1973,lusky1977separable,fonf1978massiveness,lusky-comp,
dieckmann1994korovkin,gasparis2002contractively,castillo2009extending,lusp}.
Complex $L_1$-preduals were studied  for example in \cite{wu76,nielsenolsen,wu78,roy1979convex,rao82,limaroy,rao85,kranti,utter,lusky-compl,dulin}.
It has turned out that both real and complex $L_1$-preduals can be characterized by a ``simplex-like'' property of the dual unit ball $B_{X^*}$ (see \cite{lazar,effros}). This result was used to investigate variants of the abstract Dirichlet problem for $L_1$-preduals in \cite{lusp23,lusp2,lusp-complex}.

In the present paper we study vector-valued affine functions, more precisely mappings with values in a \fr space. We show that some results remain valid in the same form, some results are true in a weaker form and some results cannot be transferred at all. Main results are contained in Section~\ref{Sec:main}. In the rest of the introductory section we collect some definitions and basic facts on compact convex sets, Baire mappings and vector integration needed in the sequel. In Section~\ref{Sec:auxiliary} we give auxiliary results needed to prove the main theorems. Some of them are known and we just collect them, some of them are up to our knowledge new and may be of an independent interest.
Several following sections are devoted to proofs of the main results. In the last section we discuss sharpness of our results and formulate related open problems.

\subsection{Compact convex sets, simplices and $L_1$-preduals.}\label{ssec:choquet}
We will deal both with real and complex spaces. To shorten the notation we will use the symbol $\ef$ to denote the respective field $\er$ or $\ce$.

If $X$ is a compact Hausdorff space, we denote by $\C(X,\ef)$ the Banach space of all $\ef$-valued continuous functions on $X$ equipped with the sup-norm. The dual of $\C(X,\ef)$ will be identified (by the Riesz representation theorem) with $\M(X,\ef)$, the space of $\ef$-valued Radon measures on $X$ equipped with the total variation norm and the respective weak$^*$ topology.

Let $X$ be a convex subset of a (real or complex) vector space $E$ and $F$ be another (real or complex) vector space. Recall that a mapping $f\colon X\to F$ is said to be \emph{affine} if $f(tx+(1-t)y)=tf(x)+(1-t)f(y)$ whenever $x,y\in X$ and $t\in[0,1]$. We stress that the notion of an affine function uses only the underlying structure of real vector spaces.

Let $X$ be a compact convex set in a Hausdorff locally convex topological vector space. We write $\fra(X,\ef)$ for the space of all $\ef$-valued continuous affine functions on $X$. This space is a closed subspace of $\C(X,\ef)$ and is equipped with the inherited sup-norm. Given a Radon probability measure $\mu$ on $X$, we write $r(\mu)$ for the \emph{barycenter of $\mu$}, i.e., the unique point $x\in X$ satisfying $a(x)=\int_X a \di\mu$ for each affine continuous function on $X$ (see \cite[Proposition~I.2.1]{alfsen} or \cite[Chapter 7, \S\,20]{lacey}; note that it does not matter whether we consider real or complex affine functions). Conversely, for a point $x\in X$, we denote by $\M_{x}(X)$ the set of all Radon probability measures on $X$ with barycenter $x$ (i.e., of probabilities \emph{representing} $x$).

The usual dilation order $\prec$ on the set $\M^1(X)$ of Radon probability measures on $X$ is defined as $\mu\prec \nu$ if and only if $\mu(f)\le \nu(f)$ for any real-valued convex continuous function $f$ on $X$. (Recall that $\mu(f)$ is a shortcut for $\int f\di\mu$.)  A measure $\mu\in \M^1(X)$ is said to be \emph{maximal} if it is maximal with respect to the dilation order.
In case $X$ is metrizable, maximal measures are exactly the probabilities carried by the $G_\delta$ set $\ext X$ of extreme points of $X$ (see, e.g., \cite[p. 35]{alfsen} or \cite[Corollary 3.62]{lmns}). By the Choquet representation theorem, for any $x\in X$ there exists a maximal representing measure (see \cite[p. 192, Corollary]{lacey} or \cite[Theorem I.4.8]{alfsen}). The set $X$ is termed \emph{simplex} if this measure is uniquely determined for each $x\in X$ (see \cite[\S 20, Theorem 3]{lacey}). We write $\delta_x$ for this uniquely determined measure.

A measure $\mu\in \M(X,\ef)$ is called \emph{boundary} if either $\mu=0$ or the probability measure $\frac{|\mu|}{\|\mu\|}$ is maximal. If $X$ is metrizable, boundary measures are exactly the measures carried by $\ext X$.

If $X$ is a simplex, the space $\fra(X,\ef)$ is an example of an $L_1$-predual (see \cite[Proposition~3.23]{fonf} for the real case; the complex case follows from the real one and \cite[\S23, Theorem 6]{lacey}). Moreover, $L_1$-preduals can be characterized using a simplex-like property of the dual unit ball:

\begin{fact}\label{fact:L1}
Let $E$ be a Banach space over $\ef$ and $X=(B_{E ^*},w^*)$. Then $E$ is an $L_1$-predual if and only if for each $x^*\in X$ there is a unique $\ef$-valued Radon measure $\mu$ on $X$ with the properties:
 \begin{itemize}
	\item[(a)] $\|\mu\|\le1$,
	\item[(b)] $\mu(\alpha A)=\ov{\alpha} \mu(A)$ for any Borel set $A\subset X$ and any $\alpha\in\ef$ with $|\alpha|=1$,
	\item[(c)] $\mu$ is a boundary measure,
	\item[(d)] $x^*(x)=\int_X y^*(x)\di\mu(y^*)$ for any $x\in E$.
\end{itemize}
\end{fact}

This fact is proved below at the end of Section~\ref{ssec:miry}. It is a variant of the characterizations given
in \cite{lazar} for the real case and in \cite{effros} for the complex case.
The unique measure $\mu$ provided by the previous fact will be denoted by $T(x^*)$.

The measures satisfying the condition (b) are called \emph{odd} in the real case and \emph{anti-homogeneous} in the complex case (we will sometimes use the unified term \emph{$\ef$-anti-homogeneous}).
The condition (d) means that $\mu$ represents $x^*$ in a way. Notice that for a function $f:X\to\ef$ the following assertions are equivalent:
\begin{itemize}
	\item[(i)] There is $x\in E$ with $f(y^*)=y^*(x)$ for $y^*\in X$.
	\item[(ii)] The function $f$ is affine, continuous and satisfies $f(\alpha y^*)=\alpha f(y^*)$ for each $y^*\in X$ and any $\alpha\in\ef$ with $|\alpha|=1$.
	\end{itemize}
Functions with the third property from the assertion (ii) are called  \emph{odd} in the real case and \emph{homogeneous} in the complex case (or \emph{$\ef$-homogeneous} in both cases).

\subsection{Baire hierarchies of mappings}\label{ssec:baire}

Given a set $K$, a topological space $L$ and a family of mappings $\F$ from $K$ to $L$, we define the \emph{Baire classes} of mappings as follows. Let $(\F)_0=\F$. Assuming  that $\alpha\in [1,\omega_1)$ is given and that $(\F)_\beta$ have
been already defined for each $\beta<\alpha$, we set
\begin{multline*}
(\F)_\alpha=\{f\colon K\to L\setsep \text{there exists a sequence } (f_n)\text{ in }\bigcup_{\beta<\alpha} (\F)_\beta
\\ \text{ such that }f_n\to f
\mbox{ pointwise}\}.
\end{multline*}

We will use several hierarchies matching this pattern:
\begin{itemize}
	\item If $K$ and $L$ are topological spaces, by $\C_\alpha(K,L)$ we denote the set $(\C(K,L))_\alpha$, where $\C(K,L)$ is the set of all continuous functions from $K$ to $L$.
	\item If $K$ is a compact convex set and $L$ is a convex subset of a locally convex space, by $\fra_\alpha(K,L)$ we denote $(\fra(K,L))_\alpha$, where $\fra(K,L)$ is the set of all affine continuous functions defined on $K$ with values in $L$.
	\item If $K=(B_{E^*},w^*)$, where $E$ is a real Banach space and $L$ is a convex symmetric subset of a  locally convex space, by $\fra_{\odd,\alpha}(K,L)$ we denote $(\fra_{\odd}(K,L))_\alpha$, where $\fra_{\odd} (K,L)$ is the set of all odd affine continuous functions defined on $K$ with values in $L$.
	\item If $K=(B_{E^*},w^*)$, where $E$ is a complex Banach space and $L$ is an absolutely convex subset of a complex locally convex space, by $\fra_{\hom,\alpha}(K,L)$ we denote $(\fra_{\hom}(K,L))_\alpha$, where $\fra_{\hom} (K,L)$ is the set of all homogeneous affine continuous functions defined on $K$ with values in $L$.
\end{itemize}

\subsection{Vector integration and strongly affine mappings}\label{ssec:pettis}
We will investigate vector-valued strongly affine mappings. To be able to do that we need some vector integral.
We will use the Pettis approach. Our vector-valued mappings  will mostly  have values in \fr spaces, but sometimes in general Hausdorff locally convex spaces. Since we will deal only with Hausdorff spaces, by a space we mean always a Hausdorff space.
Let us continue with the definitions.

Let $\mu$ be an $\ef$-valued $\sigma$-additive measure defined on an abstract measurable space $(X,\A)$ (i.e., $X$ is a set and $\A$ is a $\sigma$-algebra of subsets of $X$)
and $F$ a locally convex space over $\ef$. (To avoid confusion we stress that we will consider only finite measures.) A mapping $f:X\to F$ is said to be \emph{$\mu$-measurable} if $f^{-1}(U)$ is $\mu$-measurable for any $U\subset F$ open. The map $f$ is called  \emph{weakly $\mu$-measurable} if $\tau\circ f$ is $\mu$-measurable for each $\tau\in F^*$.

A mapping $f:X\to F$ is said to be  \emph{$\mu$-integrable} over a $\mu$-measurable set $A\subset X$ if
\begin{itemize}
\item $\tau\circ f\in L^1(|\mu|)$ for each $\tau\in F^*$,
\item for each $B\subset A$ $\mu$-measurable there exists an element $x_B\in F$ such that
\[
\tau(x_B)=\int_B \tau\circ f\di \mu,\quad \tau\in F^*.
\]
\end{itemize}
It is clear that the element $x_B$ is uniquely determined, we denote it as $\int_B f\di\mu$. If $\mu$ is clearly determined, we say only that $f$ is integrable.

Any $\mu$-integrable mapping is necessarily weakly $\mu$-measurable but not necessarily $\mu$-measurable (cf. the discussion after Question~\ref{q:kan} below).

If $X$ is a compact convex set, a mapping $f\colon X\to F$ is called \emph{strongly affine} if,
for any  measure $\mu\in\M^1(X)$, $f$ is $\mu$-integrable and $\int_X f\di\mu=f(r(\mu))$. Note that this is a 
strengthening of the notion of an affine mapping. Indeed, $f$ is affine if and only if the formula holds for any finitely
supported probability $\mu$. 

The notion of a strongly affine mapping is also a straightforward generalization of the scalar case and, in fact, the following easy observation shows a close connection with the scalar case.

\begin{fact}\label{fact1} Let $X$ be a compact convex set and $F$ a locally convex space. A mapping $f:X\to F$ is strongly affine if and only if $\tau\circ f$ is strongly affine for each $\tau\in F^*$.
\end{fact}

\begin{proof} The `only if' part follows immediately from definitions. Let us show the `if' part. Suppose that $\tau\circ f$ is strongly affine for each $\tau\in F^*$. Given any $\mu\in\M^1(X)$, $A\subset X$ $\mu$-measurable and $\tau\in F^*$,
we have
$$\int_A \tau\circ f\di\mu = \begin{cases} 0, & \mu(A)=0,\\ \mu(A) \tau\left(f\left(r\left(\frac{\mu|_A}{\mu(A)}\right)\right)\right), &\mu(A)>0.\end{cases}$$
Hence $f$ is $\mu$-integrable and $\int_X f\di\mu=f(r(\mu))$.
\end{proof}

\section{Main results}\label{Sec:main}

In this section we collect our main results. We start by results on affine Baire-one maps on general compact convex sets.
It turns out that the situation is quite different from the scalar case. We continue by several positive result which are
generalizations of theorems known in the scalar case. Sometimes we are able to generalize only weaker variants of those results. These theorems are formulated and proved in three versions -- for simplices, for dual balls of real $L_1$-preduals and for dual balls of complex $L_1$-preduals.

Let us start by considering affine Baire-one mappings.
If $X$ is a compact convex set and $f:X\to\ef$ is an affine function on $X$ which is of the first Baire class (i.e., $f\in\C_1(X,\ef)$), it is strongly affine by a result of Choquet (see, e.g., \cite[Theorem~I.2.6]{alfsen}, \cite[Section 14]{phelps-choquet}, \cite{sray} or \cite[Corollary 4.22]{lmns}) and, moreover, $f\in\fra_1(X,\ef)$ by a result of Mokobodzki (see, e.g., \cite[Th\'eor\`eme 80]{rogalski} or \cite[Theorem 4.24]{lmns}).

If $\ef=\er$, $X=(B_{E^*},w^*)$ for a real Banach space $E$ and $f$ is moreover odd, then $f\in\fra_{\odd,1}(X,\er)$. This follows easily from the general case since an affine function on $X$ is odd if and only if it vanishes at the origin.
If $\ef=\ce$, $X=(B_{E^*},w^*)$ for a complex Banach space $E$ and $f$ is moreover homogeneous, then $f\in \fra_{\hom,1}(X,\ce)$. This is a bit more difficult, it follows from Lemma~\ref{l:homaff}(c) below.

The situation for vector-valued functions is different. Firstly, the following analogue of the Choquet result follows immediately from Fact~\ref{fact1}.	

\begin{thm}\label{T:c1sa}
 Let $X$ be a compact convex set, $F$ a  locally convex space and $f:X\to F$ be an affine mapping which is of the first Baire class. Then $f$ is strongly affine.
\end{thm}

The vector version of the Mokobodzki theorem is not valid in general but it is valid under additional assumptions on the range space. This is illustrated by the following two results.

\begin{thm}\label{T:bap} Let $X$ be a compact convex set and $E$ be a Banach space with the bounded approximation property.
Then any affine $f\in\C_1(X,E)$ belongs to $\fra_1(X,E)$. If, moreover, $f(X)\subset B_E$ and $E$ has the $\lambda$-bounded approximation property for some $\lambda\ge1$, then $f\in\fra_1(X,\lambda B_E)$.
\end{thm}

\begin{example}\label{E:cap}
Let $E$ be a separable reflexive Banach space which fails the compact approximation property. Let $X=(B_E,w)$ and let $f:X\to E$ be the identity embedding. Then $f$ is affine, $f\in\C_1(X,E)$ and 
	 $f\notin\bigcup_{\alpha<\omega_1} \fra_\alpha(X,E)$.
\end{example}

This example is a strengthening of \cite[Example 2.22]{MeSta}. A Banach space satisfying the assumptions exists due to  \cite[Proposition~2.12]{casazza}.  Theorem~\ref{T:bap} is a generalization and strengthening of \cite[Theorem 2.12]{MeSta}.
We point out that the proof of the quoted theorem contains a gap. We provide a correct proof of a stronger version of the result. The two preceding results are proved in Section~\ref{sec:example} below, where also the definitions of approximation properties are recalled and the gap in the proof in \cite{MeSta} is commented.

For affine functions of higher Baire classes the situation is different even in the scalar case. Firstly, an affine function of the second Baire class need not be strongly affine even if $X$ is simplex (the example is due to Choquet, see, e.g., \cite[Example~I.2.10]{alfsen}, \cite[Section 14]{phelps-choquet} or \cite[Proposition 2.63]{lmns}). Further, by \cite{talagrand} there is a compact convex set $X$ and a strongly affine function $f:X\to\er$ of the second Baire class which does not belong to $\bigcup_{\alpha<\omega_1} \fra_\alpha(X,\er)$.
Nonetheless, some positive results hold for strongly affine functions on simplices and on dual balls of $L_1$-preduals.

We begin by the following theorem on the quality of the dilation mapping in the three cases. Let us explain the notation used in the theorem. By $\M_{\odd}(X,\er)$ we denote the space of odd real-valued Radon measures on $X$, by ${\M_{\ahom}(X,\ce)}$ the space of all anti-homogeneous complex Radon measures on $X$ (see Section~\ref{ssec:choquet} for definitions). All the range spaces in the theorem are considered with the weak$^*$ topology. 
The operator $T$ in cases (R) and (C) was defined in Section~\ref{ssec:choquet} using Fact~\ref{fact:L1}.

\begin{thm}\label{T:dilation}\
\begin{itemize}
	\item[(S)] Let $X$ be a metrizable simplex. Then the map $T:x\mapsto \delta_x$ belongs to $\fra_1(X,\M^1(X))$.
	\item[(R)] Let $E$ be a real separable $L_1$-predual and $X=(B_{E^*},w^*)$. Then the map $T$ belongs to $\fra_{\odd,1}(X,B_{\M_{\odd}(X,\er)})$.
	\item[(C)] Let $E$ be a complex separable $L_1$-predual and $X=(B_{E^*},w^*)$. Then the map $T$ belongs to $\fra_{\hom,1}(X,B_{\M_{\ahom}(X,\ce)})$.
\end{itemize}
\end{thm}

The above theorem is proved in Section~\ref{Sec:dilation}. The case (S) is essentially known, see \cite[Theorem 6.6]{lmnss03}.
The formulation of the quoted result is weaker, but the construction in fact gives the case (S). Let us point out that this result is formulated and proved only for metrizable $X$. However, in some special cases metrizability is not necessary as formulated in the following remark which will be discussed also in Section~\ref{Sec:dilation}.

\begin{remark}\label{rem:dilation} The following assertions hold even without the metrizability (separability) assumption:
\begin{itemize}
	\item $T$ is strongly affine.
	\item If $\ext X$ is closed, the mapping $T$ is continuous.
\end{itemize}
\end{remark}

We include also Example~\ref{ex:dikobraz} showing that $T$ need not be a Baire mapping even if $\ext X$ is Lindel\"of.

We continue by a result on affine Baire classes of strongly affine Baire mappings. The scalar version of the assertion (S) is proved in \cite[Theorem 2]{capon}, the scalar version of the assertion (R) follows easily from \cite[Theorem 1.4]{lusp}.
The scalar version of the assertion (C) is claimed to be unknown in \cite{lusp}. The theorem is proved in Section~\ref{Sec:affbaire}.

\begin{thm}\label{T:aff-baire} Let $X$ be a compact convex set, $F$ be a Fr\'echet space, $1\le\alpha<\omega_1$ and $f\in\C_\alpha(X,F)$ be strongly affine.
\begin{itemize}
\item[(S)] If $X$ is a simplex, then $f\in\fra_{1+\alpha}(X,F)$.
\item[(R)] If $X=(B_{E^*},w^*)$, where $E$ is a real $L_1$-predual, then $f\in\fra_{1+\alpha}(X,F)$. If $f$ is moreover odd, then $f\in\fra_{\odd, 1+\alpha}(X,F)$.
\item[(C)]  If $X=(B_{E^*},w^*)$, where $E$ is a complex $L_1$-predual, then $f\in\fra_{1+\alpha}(X,F)$. If $F$ is moreover complex and $f$ is homogeneous, then $f\in \fra_{\hom, 1+\alpha}(X, F)$.
\end{itemize}

If, moreover, $\alpha=1$, then $1+\alpha$ can be replaced by $\alpha$. I.e., if $f$ belongs to the class $\C_1$, it belongs to the class $\fra_1$.

In case $\ext X$ is an $F_\sigma$-set, $1+\alpha$ can be replaced by $\alpha$ for each $\alpha$.

\end{thm}

The next theorem is devoted to the abstract Dirichlet problem for vector-valued Baire mappings.
The scalar version of the case (S) follows from \cite{Jel}, the scalar version of the case (R) follows from \cite[Theorem 2.14]{lusp23}, the scalar version of the case (C) follows from \cite[Theorem 2.22]{lusp-complex}. The scalar versions of all the three cases hold also for $\alpha=0$, our proof of the vector version requires $\alpha\ge1$. The Lindel\"of property is quite natural assumption. It surely cannot be dropped as witnessed, for example, by the simplex from \cite[Proposition I.4.15]{alfsen} (or \cite[Example 3.82]{lmns}). However, it is still an open problem whether the Lindel\"of property is
a necessary condition for the validity of the scalar case for $\alpha=0$ (i.e., for continuous functions, see Question~\ref{q:lind} below).

\begin{thm}\label{T:dirichlet} Let $X$ be a compact convex set with $\ext X$ being Lindel\"of, $\alpha\ge 1$, $F$ a \fr space and $f:\ext X\to F$ a bounded mapping from $\C_\alpha(\ext X,F)$.
\begin{itemize}
	\item[(S)] If $X$ is a simplex, then $f$ can be extended to a mapping from $\fra_{1+\alpha}(X,F)$.
	\item[(R)] If $X=(B_{E^*},w^*)$, where $E$ is a real $L_1$-predual and $f$ is odd, then $f$ can be extended to a  mapping from $\fra_{\odd,1+\alpha}(X,F)$.
		\item[(C)] If $X=(B_{E^*},w^*)$, where $E$ is a complex $L_1$-predual, $F$ is complex and $f$ is homogeneous, then $f$ can be extended to a  mapping from $\fra_{\hom,1+\alpha}(X,F)$.
\end{itemize}

If $\ext X$ is moreover $F_\sigma$, then $1+\alpha$ can be replaced by $\alpha$ in all the cases.

If $\ext X$ is even closed and $f$ is continuous, we can find a continuous affine extension.
\end{thm}

The next theorem is devoted to the so-called `weak Dirichlet problem'. The scalar version of the case (S) is known -- for continuous functions it is proved in \cite[Theorem II.3.12]{alfsen}, for Baire functions it is due to \cite{spurny-wdp}. The cases (R) and (C) are up to our knowledge new even in the scalar case. The result is proved in Section~\ref{Sec:wdp} using a simplified and generalized variant of the method of \cite{spurny-wdp}.

\begin{thm}\label{T:weakDP}
Let $X$ be a compact convex set, $K\subset \ext X$ a compact subset, $F$ a \fr space and $f$ a bounded mapping in $\C_\alpha(K,F)$.
\begin{itemize}
\item[(S)] If $X$ is a simplex, $f$ can be extended to a mapping from $\fra_\alpha(X,\ov{\co} f(K))$.
\item[(R)] If $X=(B_{E^*},w^*)$, where $E$ is a real $L_1$-predual, $K$ is symmetric and $f$ is odd, then $f$ can be extended to a mapping from $\fra_{\odd,\alpha}(X,\ov{\aco} f(K))$.
\item[(C)] If $X=(B_{E^*},w^*)$, where $E$ is a complex $L_1$-predual, $F$ is complex, $K$ is homogeneous and $f$ is homogeneous, then $f$ can be extended to a mapping from $\fra_{\hom,\alpha}(X,\ov{\aco} f(K))$.
\end{itemize}
\end{thm}

As a consequence of this theorem we get a result on extending Baire mappings from compact subsets of completely regular spaces,
see Theorem~\ref{T:weak2}.

Finally, the following result can be viewed as an affine version of the Jayne-Rogers selection theorem. Let us recall that a set-valued mapping $\Phi$ is said to be \emph{upper semicontinuous} if $\{x\in X\setsep\Phi(x)\subset U\}$ is open in $X$ for any open set $U\subset F$ (or, equivalently $\{x\in X\setsep\Phi(x)\cap H\ne\emptyset\}$ is closed in $X$ for any closed set $H\subset F$).

\begin{thm}\label{T:selekceusc}
Let $X$ be a compact convex set, $F$ a \fr space and $\Phi\colon X\to F$ an upper semicontinuous set-valued mapping with nonempty closed values and bounded range.
\begin{itemize}
	\item[(S)] If $X$ is a metrizable simplex  and the graph of $\Phi$ is convex, $\Phi$ admits a selection in $\fra_2(X,F)$.
	\item[(R)] If $X=(B_{E^*},w^*)$, where $E$ is a separable real $L_1$-predual and the graph of $\Phi$ is convex and symmetric, $\Phi$ admits a selection in $\fra_{\odd,2}(X,F)$.
	\item[(C)] If $X=(B_{E^*},w^*)$, where $E$ is a separable complex $L_1$-predual, $F$ is complex and the graph of $\Phi$ is absolutely convex, $\Phi$ admits a selection in $\fra_{\hom,2}(X,F)$.
\end{itemize}
\end{thm}

We point out that the Jayne-Rogers selection theorem provides a selection of the first class, while we obtain a selection of the second class. This is the best we can achieve, due to Example~\ref{ex:selekce}. In the same example we show that the metrizability assumption is essential.

%%%%%%%%%%%%%%%%%%%%%%%%%%%%
%%%%%%%%%%%%%%%%%%%%%%%%%
\section{Some auxiliary results}\label{Sec:auxiliary}

Below we collect auxiliary results which we will need to prove the main results. These results are divided into four sections.
First we need the relationship between Baire hierarchy of mappings and Baire measurability. These results are known but it was necessary to collect them from the literature. The only exception is Lemma~\ref{L:c21} which is a generalization of a result of \cite{MeSta}. Further we establish some results on Pettis integration, especially a dominated convergence theorem. In the next section we collect properties of odd and homogeneous mappings and of the associated operators $\odd$ and $\hom$. The respective results are either easy or vector-valued variants of the results from \cite{lusp-complex}.
Finally, we investigate adjoint operators to $\odd$ and $\hom$ and odd and antihomogeneous measures. We think that some of these results are new and of an independent interest. In particular, we prove there Fact~\ref{fact:L1} which is a simplex-like characterization of $L_1$-preduals.

\subsection{Baire hierarchy of sets and Baire mappings}
In this section we formulate the exact relationship between Baire mappings and mappings measurable with respect to the Baire $\sigma$-algebra.

If $X$ is a topological space, a \emph{zero set} in $X$ is the  inverse image of a closed set in $\er$ under a continuous function $f:X\to\er$. The complement of a zero set is a \emph{cozero set}. If $X$ is normal, it follows from Tietze's theorem that a closed set is a zero set if and only if it is also a $G_\delta$ set. We recall that \emph{Baire sets} are members of the $\sigma$-algebra generated by  the family of all cozero sets in $X$.

We will need a precise hierarchy of Baire sets.
%Let $\Baire(K)$ be the algebra generated by zero sets in $K$, i.e., the system of finite unions of differences of zero sets.
%\[
%\Baire(K)=\{\bigcup_{i=1}^n (F_i\setminus H_i): F_i, H_i\text{ are zero sets in }K, n\in\en\}.
%\]
We define additive and multiplicative Baire classes of sets as follows:
Let $\Sigma_1^b(X)$ be the family of all cozero sets and $\Pi_1^b(X)$ the family of all zero sets. For $\alpha\in (1,\omega_1)$, let
\begin{itemize}
	\item $\Sigma_\alpha^b(X)$ be the family of countable unions of sets from $\bigcup_{\beta<\alpha}\Pi_\beta^b(X)$, and
	\item $\Pi_\alpha^b(X)$ be the family of countable intersections of sets from $\bigcup_{\beta<\alpha}\Sigma_\beta^b(X)$.
\end{itemize}
The family $\Sigma_\alpha^b(X)$ is termed the \emph{sets of Baire additive class $\alpha$},
the family $\Pi_\alpha^b(X)$ is called the \emph{sets of Baire multiplicative class $\alpha$}.

The following two lemmata collect some properties of Baire measurable mappings useful for our investigation.

\begin{lemma}\label{L:baire-n}
Let $X$ be a topological space and $F$ be a metrizable separable space. Let $f\colon X\to F$ be a Baire measurable mapping. Then the following assertions hold.
\begin{itemize}
	 \item [(a)] The $\sigma$-algebra of Baire subsets of $X$ equals $\bigcup_{\alpha<\omega_1} \Sigma_\alpha^b(X)=\bigcup_{\alpha<\omega_1}\Pi_\alpha^b(X)$.
	 \item [(b)] There exists $\alpha<\omega_1$ such that $f$ is $\Sigma_\alpha^b(X)$-measurable.
	 \item [(c)] If $X$ is normal, $F$ is a convex subset of a \fr space and $\alpha\in [0,\omega_1)$, then $f\in\C_\alpha(X,F)$ if and only if $f$ is $\Sigma_{\alpha+1}^b(X)$-measurable.
\end{itemize}
\end{lemma}

\begin{proof} The assertion (a) is obvious.

(b) Since separable metric spaces have countable basis, the assertion easily follows.

(c) For $\alpha=0$ the assertion is trivial. For $\alpha=1$ it follows from \cite[Theorem 3.7(i)]{vesely}.

The assertion for $\alpha>1$ follows from \cite[Theorem 2.7]{spurny-amh}. Indeed, if $\F$ denotes the algebra generated by zero sets, then the families $\Sigma_\alpha(\F)$ and $\Pi_\alpha(\F)$ from the quoted paper are exactly $\Sigma_\alpha^b(X)$ and $\Pi_\alpha^b(X)$ for $\alpha\ge2$. Further, the family denoted by $\Phi_\alpha$ in the quoted paper is exactly $\C_\alpha(X,F)$
for $\alpha\ge1$. For $\alpha=1$ it follows from the previous paragraph, the validity for larger ordinals follows from the definitions.
\end{proof}

An immediate consequence of the assertion (c) of the previous lemma is the following statement.

\begin{cor}\label{C:baire} Let $X$ be a normal space, $F$ a separable \fr space, $\alpha<\omega_1$ and $f\in\C_\alpha(X,F)$.
Then $f\in\C_\alpha(X,\co f(X))$.
\end{cor}

\begin{lemma}
\label{L:baire}
Let $X$ be a Baire subset of a compact space and $F$ be a metrizable space. Let $f\colon X\to F$ be a Baire measurable mapping. Then the following assertions hold.
\begin{itemize}
\item [(a)] The image $f(X)$ is separable.
\item [(b)] There exists $\alpha<\omega_1$ such that $f$ is $\Sigma_\alpha^b(X)$-measurable.
\item [(c)] If $F$ is a convex subset of a \fr space and $\alpha\in [0,\omega_1)$, then $f\in\C_\alpha(X,F)$ if and only if $f$ is $\Sigma_{\alpha+1}^b(X)$-measurable.
\end{itemize}
\end{lemma}

\begin{proof}
The assertion (a) follows from \cite[Theorem 1]{frolik-bulams}. Indeed, the space $X$, being a Baire subset of a compact space,
is $K$-analytic by \cite[Proposition $\beta$ on p.1113]{frolik-bulams} (spaces which are now called $K$-analytic are called \emph{analytic} in the quoted paper). Thus $X$ satisfies the assumption on the domain space in \cite[Theorem 1]{frolik-bulams}.
Further, $f$ is Baire measurable. It should be noted, that in \cite{frolik-bulams} this notion has a different meaning
-- it means that the preimage of any Baire set is a Baire set. Since $F$ is assumed to be metrizable, this notion coincides with our notion of Baire measurability. The conclusion is that $f(X)$ is $K$-analytic, hence Lindel\"of (by \cite[Proposition $\alpha$ on p.1113]{frolik-bulams}) and therefore separable (by metrizability).

(b) By (a) we can suppose without loss of generality that $F$ is separable. Hence the assertion follows from Lemma~\ref{L:baire-n}(b).

(c) By (a) we can suppose without loss of generality that $F$ is separable. Since $X$ is normal (being regular and Lindel\"of), the assertion follows from Lemma~\ref{L:baire-n}(c).
\end{proof}

Let us point out that the assertion (a) of the previous lemma is not valid in general. It is valid (by the same proof) under a weaker assumption that $X$ is $K$-analytic (but we do not need it), but it fails if $X$ is, say, a general separable metric space. An example under Martin's axiom and negation of the continuum hypothesis is described in \cite[Example 2.4(3)]{koumou}.

We continue by the following technical lemma which is a variant of \cite[Proposition 2.8]{MeSta}. In the quoted paper the authors deal with mappings with values in a Banach space. Our proof essentially follows the proof in \cite{MeSta} with necessary modifications due to the assumption that $F$ is a \fr space, not necessarily a Banach space.

\begin{lemma}\label{L:c21}
Let $X$ be a compact space, $F$ a \fr space over $\ef$, $f_{n,m}$, $f_n$, $f$ ($m,n\in\en$) be mappings defined on $X$ with values in $F$ satisfying the following conditions:
\begin{itemize}
	\item[(i)] $f_{n,m}$ is continuous on $X$ for each $n,m\in\en$;
	\item[(ii)] $f_{n,m}(x)\overset{m}{\longrightarrow}f_n(x)$ weakly in $F$ for each $n\in\en$ and each $x\in X$;
	\item[(iii)] $f_{n}(x)\overset{n}{\longrightarrow}f(x)$ weakly in $F$ for each $x\in X$;
	\item[(iv)] the family of functions $(f_{n,m})$ is uniformly bounded;
	\item[(v)] $f\in\C_1(X,F)$.
\end{itemize}
Then there is a sequence $(g_k)$ of convex combinations of functions $f_{n,m}$, $n,m\in\en$, such that $g_k(x)\to f(x)$
in $F$ for each $x\in X$.
\end{lemma}

\begin{proof} Fix a closed absolutely convex bounded set $L\subset F$ containing the ranges of all the functions $f_{n,m}$, $n,m\in\en$. Since closed convex sets are weakly closed, it contains also ranges of $f_n$, $n\in\en$, and that of $f$.
Hence by Corollary~\ref{C:baire} we have $f\in\C_1(X,L)$ (note that $L$ can be chosen to be separable). Fix a sequence $(h_k)$ in $\C(X,L)$ pointwise converging to $f$.

Since $F$ is a \fr space, its topology is generated by a sequence of seminorms $(p_k)$. Without loss of generality we can
suppose that $p_1\le p_2\le \dots$. For each $k\in\en$ denote by $K_k$ the polar of the set $\{y\in F\setsep p_k(y)<1\}$. By the Alaoglu theorem $K_k$ is weak$^*$ compact. 

Given any function $u\colon X\to F$ denote by $\widehat{u}$ the scalar function on $X\times F^*$ defined by
$$ \widehat{u}(x,x^*)=x^*(u(x)), \qquad (x,x^*)\in X\times F^*.$$
It is easy to observe that
\begin{itemize}
	\item $\widehat{u_n}\to \widehat{u}$ pointwise whenever $u_n(x)\to u(x)$ weakly in $F$ for each $x\in X$.
	\item If $u$ is continuous, then $\widehat{u}$ is continuous on $X\times K_k$ for each $k\in\en$. ($K_k$ is equipped with the weak$^*$ topology.)
\end{itemize}
Indeed, the first assertion is obvious. Let us show the second one. Suppose $u$ is continuous and $k\in\en$.
Fix any $(x,x^*)\in X\times K_k$ and $\varepsilon>0$. For $(y,y^*)\in X\times K_k$ we have
$$\begin{aligned}\abs{\widehat{u}(y,y^*)-\widehat{u}(x,x^*)}&=\abs{y^*(u(y))-x^*(u(x))}
\\& \le \abs{y^*(u(y)-u(x))}+\abs{y^*(u(x))-x^*(u(x))}
\\& \le p_k(u(y)-u(x))+\abs{y^*(u(x))-x^*(u(x))}<\varepsilon\end{aligned}$$
whenever
$$p_k(u(y)-u(x))<\frac\varepsilon2\mbox{\quad and \quad}\abs{y^*(u(x))-x^*(u(x))}<\frac\varepsilon2,$$
which defines a neighborhood of $(x,x^*)$ in $X\times K_k$.

Fix $k\in \en$. Then $\widehat{h_n}\to \widehat{f}$ pointwise and hence $\widehat{f}\,\r_{X\times K_k}\in\C_1(X\times K_k,\ef)$. Moreover, $\widehat{f_n}\to \widehat{f}$ and $\widehat{f_{n,m}}\overset{m}{\longrightarrow} \widehat{f_n}$  pointwise on $X\times K_k$ for each $n\in\en$. Since  $(\widehat{f_{n,m}}\r_{X\times K_k})$ is a uniformly bounded family of continuous functions, \cite[Lemma 2.5]{MeSta} yields a sequence $(u^k_n)_n$ in the convex hull of the family $(f_{n,m})$ such that
$\widehat{u^k_n}\overset{n}{\longrightarrow}\widehat{f}$ pointwise on $X\times K_k$. Then $\widehat{u^k_n}-\widehat{h_n}\overset{n}{\longrightarrow}0$ pointwise on $X\times K_k$. Since this sequence is uniformly bounded in $\C(X\times K_k,\ef)$, it converges weakly to zero in $\C(X\times K_k,\ef)$. By the Mazur theorem there is
$v_k\in\co\{u^k_n-h_n\setsep n\ge k\}$ such that $\norm{\widehat{v_k}\r_{X\times K_k}}<\frac1k$. The function $v_k$ can be expressed as $v_k=g_k-w_k$, where $g_k\in\co\{f_{n,m}\setsep n,m\in\en\}$ and $w_k\in\co\{h_n\setsep n\ge k\}$.

We claim that $g_n(x)\to f(x)$ for each $x\in X$. So, fix $x\in X$ and $k\in \en$. We will show that $p_k(f(x)-g_n(x))\overset{n}{\longrightarrow} 0$.
Let $\varepsilon>0$ be arbitrary. Fix $n_0\in\en$ such that $n_0\ge k$, $\frac{1}{n_0}<\frac\varepsilon2$ and
	 $p_k(f(x)-h_n(x))<\frac\varepsilon2$ for $n\ge n_0$. Fix $n\ge n_0$. Then $p_k(f(x)-w_n(x))<\frac\varepsilon2$ 
	 (as $w_n(x)\in\co\{h_j(x)\setsep j\ge n_0\}$). Hence
$$\begin{aligned}p_k(f(x)-g_n(x))&\le p_k(f(x)-w_n(x))+p_k(w_n(x)-g_n(x)) \le \frac\varepsilon2 + p_n(v_n(x))
\\&=\frac\varepsilon2+\sup\{\abs{x^*(v_n(x))}\setsep x^*\in K_k\} 
=\frac\varepsilon2+\norm{\widehat{v_n}\r_{X\times K_k}}
\\&\le\frac\varepsilon2+\norm{\widehat{v_n}\r_{X\times K_n}} <\frac\varepsilon2+\frac1n<\varepsilon.\end{aligned}$$
This completes the proof.
\end{proof}

An immediate consequence of the previous lemma is the following statement.

\begin{cor}\label{C:c21} Let $X$ be a compact convex space and $F$ be a bounded convex subset of a \fr space.
If $f\in\C_1(X,F)$ and $f\in\fra_2(X,F)$, then $f\in \fra_1(X,F)$.
\end{cor}

\subsection{Integrable vector-valued functions}
We collect several results on vector integration needed in the sequel. We start by two lemmata on the relationship
between measurability and weak measurability.

\begin{lemma}
\label{L:meas-0}
Let $F$ be a separable metrizable locally convex space. Then each open subset of $F$ is $F_\sigma$ in the weak topology.
\end{lemma}

\begin{proof}
Let $U$ be an open subset of $F$. For each $x\in U$ there is a convex open neighborhood $V_x$ of zero such that $x+\overline{V_x}\subset U$.
Since $F$ is metrizable and separable, there is a countable set $C\subset U$ such that $x+V_x$, $x\in C$, cover $U$. Then $U=\bigcup_{x\in C}x+\overline{V_x}$.
Since closed convex sets are weakly closed, the proof is completed.
\end{proof}

\begin{lemma}
\label{L:meas-ws}
Let $\mu$ be an $\ef$-valued measure defined on a measurable space $(X,\A)$ and $F$ be a separable metrizable locally convex space. Then any weakly $\mu$-measurable function $f:X\to F$ is $\mu$-measurable.
\end{lemma}

\begin{proof} Suppose that $f$ is weakly measurable, i.e., $\tau\circ f$ is $\mu$-measurable for any $\tau\in F^*$.
It follows that $f^{-1}(U)$ is $\mu$-measurable for any set $U$ from the canonical basis of the weak topology on $F$.
Since $F$ is separable and metrizable, the weak topology is hereditarily Lindel\"of, thus $f^{-1}(U)$ is $\mu$-measurable for any weakly open set $U$. By Lemma~\ref{L:meas-0} we conclude that $f^{-1}(U)$ is $\mu$-measurable for any set $U$ open in the original topology of $F$. This completes the proof.
\end{proof}

\begin{lemma}
\label{L:norma}
Let $\mu$ be an $\ef$-valued measure defined on a measurable space $(X,\A)$ and $F$ be a \fr space over $\ef$.
Suppose that $f\colon X\to F$ is a bounded weakly $\mu$-measurable mapping with (essentially) separable range.
Then the following assertions hold.
\begin{itemize}
\item [(a)] The mapping $f$ is  $\mu$-integrable.
\item [(b)] If $\mu$ is moreover a probability and $L\subset F$ is a closed convex set such that $f(X)\subset L$, then $\mu(f)\in L$.
\item [(c)] If $\norm{\mu}\le1$ and $L\subset F$ is a closed absolutely convex set such that $f(X)\subset L$, then $\mu(f)\in L$.

\item [(d)] If $\rho$ is any continuous seminorm on $F$, then $\rho\circ f$ is $\mu$-integrable and $\rho\left(\int_X f\di \mu\right)\le \int_X \rho\circ f\di\abs{\mu}$.
\end{itemize}
\end{lemma}

\begin{proof} We can without loss of generality suppose that $F$ is separable.

The assertion (a) for nonnegative measures then follows immediately from  \cite[Corollary~3.1]{thomas}.
The general case is an easy consequence.

(b) Assuming $\mu(f)\notin L$, we can by the Hahn-Banach separation argument find an element $\tau\in F^*$ and $c\in\er$ such that  $\Re\tau(\mu(f))>c>\sup \{\Re \tau(l)\setsep l\in L\}$.
Then
\[
c<\Re\tau(\mu(f))=\int_X \Re \tau(f(x))\di\mu(x)< \int_X c\di\mu(x)=c.
\]
Hence the assertion follows.

(c) We proceed in the same way as in the proof of (b). If $L$ is absolutely convex, we get
$$\sup \{\Re \tau(l)\setsep l\in L\}=\sup \{\abs{\tau(l)}\setsep l\in L\}$$
and
\[
c<\Re\tau(\mu(f))=\Re \int_X  \tau(f(x))\di\mu(x)\le\int_X \abs{\tau(f(x))}\di\abs{\mu}(x) \le c.
\]

To show (d), let $\rho$ be any continuous seminorm on $F$. Since $f$ is $\mu$-measurable by Lemma~\ref{L:meas-ws}, it is clear that $\rho\circ f$ is $\mu$-measurable and bounded, hence it is $\mu$-integrable. Set $V=\{x\in F\setsep\rho(x)\le 1\}$ and let
\[
V^0=\{\tau\in F^*\setsep \abs{\tau(x)}\le 1\mbox{ for }x\in V\}
\]
denote the absolute polar of $V$. Then the Bipolar Theorem implies that
\[
\rho(x)=\sup\{\abs{\tau(x)}\setsep \tau\in V^0\},\quad x\in F.
\]
Hence
\[
\begin{aligned}
\rho\left(\int_X f\di\mu\right)&=\sup\left\{\abs{\tau\left(\int_X f\di\mu\right)}\setsep \tau\in V^0\right\}
=\sup\left\{\abs{\int_X \tau\circ f\di\mu}\setsep \tau\in V^0\right\}\\
&\le \sup\left\{\int_X \abs{\tau\circ f}\di\abs{\mu}\setsep \tau\in V^0\right\}
\le \int_X \rho\circ f\di\abs{\mu}.
\end{aligned}
\]
This concludes the proof.

\end{proof}

An important class of integrable functions are Baire measurable functions.

\begin{lemma}
\label{L:integr-for-baire}
Let $X$ be a compact space, $\mu$ an $\ef$-valued Radon measure on $X$ and $f\colon X\to F$ be a bounded Baire measurable mapping from $X$ to a \fr space $F$ over $\ef$. Then the mapping $f$ is  $\mu$-integrable.
\end{lemma}

\begin{proof}
By Lemma~\ref{L:baire}(a) the image $f(X)$ is separable. Hence the conclusion follows from Lemma~\ref{L:norma}(a).
\end{proof}

We will use also the following version of the Dominated Convergence Theorem.

\begin{thm}[Dominated Convergence Theorem]
\label{T:dct}
Let $\mu$ be an $\ef$-valued measure defined on a measurable space $(X,\A)$ and $F$ be a \fr space over $\ef$.
Let $f_n,f\colon X\to F$ be mappings such that
\begin{itemize}
\item $f_n$ are weakly $\mu$-measurable and have separable range,
\item the sequence $\{f_n\}$ is bounded in $F$ (i.e., $\bigcup_{n=1}^\infty f_n(X)$ is bounded in $F$),
\item $f_n(x)\to f(x)$ in $F$ for $x\in X$.
\end{itemize}
Then $f$ is bounded and $\mu$-measurable. Moreover, all the involved functions are $\mu$-integrable and $\int_X f_n\di\mu\to \int_X f\di\mu$ in $F$.
\end{thm}

\begin{proof} Set $L=\ov{\{f_n(x)\setsep n\in\en, x\in X\}}$. Then $L$ is a separable closed bounded set and clearly $f(X)\subset L$. Thus $f$ is bounded and has separable range. Moreover, it is weakly $\mu$-measurable as measurability is preserved by pointwise limits of sequences. By Lemma~\ref{L:meas-ws}, all the involved mappings are even $\mu$-measurable. Their $\mu$-integrability now follows from Lemma~\ref{L:norma}(a). It remains to prove the convergence of integrals.

To do that, let $\rho$ be any continuous seminorm on $F$. By Lemma~\ref{L:norma}(c) we get
$$\rho\left(\int_X f_n\di\mu-\int_X f\di\mu\right)=\rho\left(\int_X (f_n-f)\di\mu\right)\le \int_X\rho\circ(f_n-f)\di\abs{\mu}$$
for each $n\in\en$. Since $f_n\to f$ pointwise, $\rho\circ(f_n- f)\to 0$ pointwise.
Clearly $\rho(f_n(x)-f(x))\le 2\sup\{\rho(v)\setsep v\in L\}$ for each $x\in K$,
thus the classical Dominated Convergence Theorem gives
\[
\lim_{n\to \infty}\int_X \rho\circ(f_n-f)\di\abs{\mu}=0.
\]
Since $\rho$ is an arbitrary continuous seminorm, the convergence of the integrals follows.
\end{proof}

\subsection{Odd and homogeneous mappings}

Let $E$ be a vector space over $\ef$. A set $A\subset E$ is called \emph{$\ef$-homogeneous} if $\alpha x\in A$ whenever $x\in A$ and $\alpha\in \ef$ satisfies $|\alpha|=1$. $\er$-homogeneous sets are called \emph{symmetric}, $\ce$-homogeneous sets just \emph{homogeneous}.

If $A\subset E$ is $\ef$-homogeneous, $F$ is another vector space over $\ef$, then a mapping $f:A\to F$ is called \emph{$\ef$-homogeneous} if $f(\alpha x)=\alpha f(x)$ whenever $x\in A$ and $\alpha\in \ef$ satisfies $|\alpha|=1$. $\er$-homogeneous maps are called \emph{odd}, $\ce$-homogeneous maps just \emph{homogeneous}.

The following lemma shows that there is a close connection between $\ef$-homo\-ge\-neous affine maps and linear operators.

\begin{lemma}
\label{l:operator}
Let $E$ be a Banach space over $\ef$ and let $f\colon B_{E}\to F$ be an $\ef$-homogeneous affine mapping from $B_E$ to a vector space $F$ over $\ef$. Then there exists a unique linear operator $L\colon E\to F$ extending $f$.
\end{lemma}

\begin{proof}
We define the required extension $L\colon E\to F$ by the formula
\[
Lx=\begin{cases} \norm{x}f(\frac{x}{\norm{x}}),& \norm{x}>1,\\
                 f(x),&\norm{x}\le 1.
    \end{cases}
\]
Then $L$ is a well defined linear extension of $f$ which is obviously unique.
\end{proof}

We are now going to describe a procedure of ``homogenization'' of functions. This procedure is easier in the real case.

Let $E$, $F$ be  real vector spaces and $B\subset E$ a symmetric set. For any function $f:B\to F$ we define a function $\odd f$ by the formula
\[
(\odd f)(x)=\frac12(f(x)-f(-x)),\quad x\in B.
\]
It is clear that $\odd f$ is an odd function and that $f$ is odd if and only if $f=\odd f$.

The complex procedure is a bit more involved. Again, we suppose that $E$ and $F$ are complex vector spaces, $B\subset E$ is a homogeneous set and $f:B\to F$ is a mapping. We would like to define a function $\hom f$ by the formula
\[
(\hom f)(x)=\frac1{2\pi}\int_0^{2\pi} e^{-it}f(e^{it} x)\di t, \quad x\in B.
\]
The difference from the real case is that a vector integral is involved in the formula and it need not have a sense. Therefore we restrict the assumptions -- we will assume that  $F$ is locally convex space and that the respective integral is well defined (in the Pettis sense defined in Section~\ref{ssec:pettis}) for each $x\in B$. These assumptions are satisfied, in particular, if $E$ is a locally convex space, $F$ is a \fr space and $f$ is a bounded Borel function.
Indeed, then for each $x\in B$ the mapping $t\mapsto e^{-it}f(e^{it} x)$ is a bounded Borel mapping from $[0,2\pi]$ to $F$. Since for compact metric spaces Baire and Borel $\sigma$-algebras coincide, we can conclude by Lemma~\ref{L:integr-for-baire}.
It is clear that $\hom f$ is homogeneous whenever it has a sense.

The following lemma sums up basic properties of the operator $\odd$.

\begin{lemma}\label{l:odd}
Let $E$, $F$ be real locally convex spaces and $A\subset E$, $B\subset F$ be symmetric sets.
\begin{itemize}
	\item[(a)] The function $\odd f$ is continuous for each $f:A\to F$ continuous.
	\item[(b)] If $f\in\C_\alpha(A,B)$ for some $\alpha<\omega_1$, then $\odd f\in\C_{\odd,\alpha}(A,B)$ where $\C_{\odd,\alpha}(A,B)=\left(\C_{\odd}(A,B)\right)_\alpha$ and $\C_{\odd}(A,B)$ is the space of all odd continuous mappings of $A$ into $B$.
		\item[(c)] If $A$ and $B$ are moreover convex and $f\in\fra_{\alpha}(A,B)$ for some $\alpha<\omega_1$, then $\odd f\in\fra_{\odd,\alpha}(A,B)$.
	\item[(d)] If $A$ and $B$ are moreover convex, then each odd function from $\fra_\alpha(A,B)$ belongs to $\fra_{\odd,\alpha}(A,B)$.
\end{itemize}
\end{lemma}

\begin{proof} The assertion (a) is obvious. The assertion (b) follows from (a) by transfinite induction.
The assertion (c) is obvious for $\alpha=0$, the general case follows by transfinite induction.
Finally, (d) is an immediate consequence of (c).
\end{proof}

We continue by basic properties of the operator $\hom$. The scalar version is proved in \cite[Lemma 2.2]{lusp-complex}.

\begin{lemma}
\label{l:hom}
Let $E$ be a complex locally convex space, $F$ a complex \fr space, $A\subset E$ a homogeneous set and $B\subset F$ a closed absolutely convex bounded set. Then the following assertions hold.
\begin{itemize}
\item [(a)] The function $\hom f$ is continuous for each $f:A\to B$ continuous.
\item [(b)] Let $f_n:A\to B$ be a Borel function for each $n\in\en$. Suppose that the sequence $(f_n)$ pointwise converges to a function $f$. Then the sequence $(\hom f_n)$ pointwise converges to $\hom f$.
\item [(c)] If $f\in\C_\alpha(A,B)$ for some $\alpha<\omega_1$, then $\hom f\in\C_{\hom,\alpha}(A,B)$  where $\C_{\hom,\alpha}(A,B)=\left(\C_{\hom}(A,B)\right)_\alpha$ and $\C_{\hom}(A,B)$ is the space of all homogeneous continuous mappings of $A$ into $B$.
\end{itemize}
\end{lemma}

\begin{proof} (a) Let $x\in A$ be arbitrary. Fix any continuous seminorm $\rho$ on $F$ and any $\ep>0$. We will find a neighborhood $U$ of zero in $E$ such that for any $y\in (x+U)\cap A$ one has $\rho(\hom f(y)-\hom f(x))\le\ep$.

To do that first observe that the mapping $h:A\times\er\to B$ defined by $h(y,t)=e^{-it}f(e^{it}y)$ is continuous. Therefore for any $t\in \er$ there is $V_t$, an absolutely convex neighborhood of zero in $E$, and $\delta_t>0$ such that whenever $s\in(t-\delta_t,t+\delta_t)$ and $y\in (x+V_t)\cap A$, then $\rho(h(y,s)-h(x,t))<\frac\ep2$. Fix a finite set $J\subset[0,2\pi]$ such that the intervals $(t-\delta_t,t+\delta_t)$, $t\in J$, cover
$[0,2\pi]$. Let $U=\bigcap_{t\in J}V_t$. Given $s\in[0,2\pi]$, choose $t\in J$ with $s\in(t-\delta_t,t+\delta_t)$. Then for each $y\in (x+U)\cap A$ we have
$$\rho(h(y,s)-h(x,s))\le\rho(h(y,s)-h(x,t))+\rho(h(x,t)-h(x,s))<\ep.$$
Hence, for each $y\in (x+U)\cap A$ we have
$$\begin{aligned}
\rho(\hom f(y)-\hom f(x))&=\rho\left(\frac1{2\pi}\int_0^{2 \pi} (h(y,s)-h(x,s))\di s\right)
\\&\le \frac1{2\pi}\int_0^{2 \pi} \rho(h(y,s)-h(x,s))\di s\le\ep,
\end{aligned}$$
which concludes the proof.

The assertion (b) follows immediately from Theorem~\ref{T:dct}. The assertion (c) follows from (a) using (b) and transfinite induction.
\end{proof}

In case $f$ is affine, $\hom f$ is always well defined and no further measurability assumptions on $f$ are needed. Indeed, given $x\in E$, the set $\{e^{it}x\setsep t\in[0,2\pi]\}$ is contained in a finite-dimensional subspace of $E$ (more precisely, in a subspace of complex dimension one, hence of real dimension two). Moreover, $f$ is continuous on this finite-dimensional space and maps it into a finite-dimensional subspace of $F$.
The properties of the ``homogenization'' of affine mappings are summed up in the following lemma.

\begin{lemma} Let $E$, $F$ be complex locally convex spaces, $A\subset E$ an absolutely convex set and $B\subset F$ a closed absolutely convex set.
\label{l:homaff}
\begin{itemize}
\item [(a)] If $f:A\to F$ is affine, then $$\hom f(x)=\frac12(f(x)-if(ix))-\frac{1-i}2f(0),\quad x\in A.$$
Moreover, there are unique homogeneous affine functions $u,v:A\to F$ such that $f(x)=f(0)+u(x)+\ov{v(x)}$ for $x\in A$.
In this formula $u=\hom f$.
\item [(b)] If $f\in\fra_{\alpha}(A,B)$ for some $\alpha<\omega_1$, then $\hom f\in\fra_{\hom,\alpha}(A,B)$.
\item [(c)] Each homogeneous function from $\fra_{\alpha}(A,B)$ belongs to $\fra_{\hom,\alpha}(A,B)$ as well.
\end{itemize}
\end{lemma}

\begin{proof} (a) Set $h=f-f(0)$. Then $h$ is odd, hence $h(\alpha x+\beta y)=\alpha h(x)+\beta h(y)$ whenever $\alpha,\beta\in \er$ and $x,y,\alpha x+\beta y\in A$. Indeed, this is clear in case $\beta=0$. If $\beta\ne 0$, we have
$$\begin{aligned}h(\alpha x+\beta y)&=2(\abs{\alpha}+\abs{\beta})h\left(\frac{\alpha x+\beta y}{2(\abs{\alpha}+\abs{\beta})}\right)\\&
=(\abs{\alpha}+\abs{\beta})\left(h\left(\frac{\alpha}{\abs{\alpha}+\abs{\beta}}x\right)+h\left(\frac{\beta}{\abs{\alpha}+\abs{\beta}}y\right)\right)=\alpha h(x)+\beta h(y).\end{aligned}$$
 Therefore we have
$$\begin{aligned}
\hom f(x)&=\hom h(x)=\frac1{2\pi}\int_0^{2\pi}e^{-it}h(e^{it}x)\di t
\\&=\frac1{2\pi}\int_0^{2\pi} e^{-it}h(x\cos t+ix\sin t)\di t
\\&=\frac1{2\pi}\int_0^{2\pi} (\cos^2t-i\cos t\sin t)(h(x)) \di t
\\&\qquad\qquad+\frac1{2\pi}\int_0^{2\pi} (\cos t\sin t-i\sin^2t)h(ix)\di t
\\&=\frac12( h(x) - ih(ix))
=\frac12(f(x)-if(ix))-\frac{1-i}2f(0).
\end{aligned}$$
Set $u=\hom f$ and $v(x)=\ov{f(x)-f(0)-u(x)}$ for $x\in A$. Then $u$ is a homogeneous affine function.
Moreover, by the above formula we get
$$v(x)=\frac12\ov{(f(x)-f(0))+i(f(ix)-f(0))}, \quad x\in A.$$
 Hence
 $v(ix)=iv(x)$ for $x\in A$, so $v$ is homogeneous. The uniqueness is clear.

The assertion (b) follows from (a) by transfinite induction.
The assertion (c) follows immediately from (b).
\end{proof}

\subsection{Odd and anti-homogeneous measures}\label{ssec:miry}

Let $E$ be a Banach space over $\ef$ and $X=(B_{E^*},w^*)$. Then $X$ is a compact convex set. Odd and anti-homogeneous measures on $X$ were defined in Section~\ref{ssec:choquet}. In this section we will elaborate these notions and, in particular, provide a proof of Fact~\ref{fact:L1}.

There are two points of view on a Radon measure on $X$ -- we can view it as a set function or as a functional on $\C(X,\ef)$.
In the previous section we defined operators $\odd$ and $\hom$ on the space $\C(X,\ef)$. We use the same symbols to denote
the adjoint operator on $\C(X,\ef)^*$. Let $\mu$ be an $\ef$-valued Radon measure on $X$.
 \begin{itemize}
	\item If $\ef=\er$, we define $\odd\mu\in \C(X,\er)^*$ by $(\odd\mu)(f)=\mu(\odd f)$, $f\in\C(X,\er)$.
	\item If $\ef=\ce$, we define $\hom\mu\in \C(X,\ce)^*$ by $(\hom \mu)(f)=\mu(\hom f)$, $f\in \C(X,\ce)$.
\end{itemize}

The following lemma sums up basic properties of the operator $\odd$ on measures.

\begin{lemma}\label{l:odd-miry}
Let $E$ be a real Banach space, $X=(B_{E^*},w^*)$ and $\mu$ be a signed Radon measure on $X$.
Then the following assertions hold.
\begin{itemize}
	\item[(a)] $\odd\mu(A)=\frac12(\mu(A)-\mu(-A))$ for any  Borel set $A\subset X$.
	\item[(b)] $\odd\mu$ is an odd measure.
	\item[(c)] $\mu$ is odd if and only if $\mu=\odd\mu$.
	\item[(d)] Let $B\subset X$ be a symmetric $\mu$-measurable set and $f:B\to \er$ a bounded Borel function. Then $\int_B f\di\odd\mu=\int_B\odd
 f\di\mu$.
 \item[(e)] If $\mu$ is boundary, then $\odd\mu$ is boundary.
\end{itemize}
\end{lemma}

\begin{proof} (a) Define a measure $\nu$ by $\nu(A)=\frac12(\mu(A)-\mu(-A))$ for $A\subset X$ Borel. It is a well-defined measure. Moreover, for any $f\in\C(X,\er)$ we have $\nu(f)=\mu(\odd f)$. Indeed, set $\sigma(x^*)=-x^*$ for $x^*\in X$. Then, given $f\in\C(X,\er)$, we have
$$\nu(f)=\frac12\left(\int_X f\di\mu-\int_X f\di\sigma(\mu)\right)=\frac12\left(\int_X f\di\mu-\int_X f\circ \sigma^{-1}\di\mu\right)=\mu(\odd f).$$
It follows that $\nu=\odd\mu$.

The assertions (b) and (c) follow immediately from (a). The assertion (d) follows from (a) by repeating the above computation.
The assertion (e) follows easily from the characterization of boundary measures given in \cite[p. 34--35]{alfsen}.
 \end{proof}

We continue by the following lemma on odd measures and their relationship to probabilities.

\begin{lemma}\label{l:odd-miry2}
Let $E$ be a real Banach space, $X=(B_{E^*},w^*)$ and $\mu$ be an odd Radon measure on $X$ with $\|\mu\|\le 1$.
Then the following assertions hold.
\begin{itemize}
	\item[(a)] There is a Borel set $A\subset X$ with $A\cap(-A)=\emptyset$ such that $\mu^+$ is supported by $A$ and $\mu^-$ is supported by $-A$.
	\item[(b)] There is a unique $x^*\in X$ such that for each $x\in E$ one has $x^*(x)=\int y^*(x)\di\mu(y^*)$.
	\item[(c)] There is a probability $\nu$ on $X$ such that $\odd\nu=\mu$ and $r(\nu)=x^*$ (where $x^*$ is provided by (b)).
	If $\mu$ is boundary, $\nu$ can be chosen maximal.
\end{itemize}
\end{lemma}

\begin{proof} (a) Let us fix two disjoint Borel sets $A^+$ and $A^-$ such that $\mu^+$ is supported by $A^+$ and $\mu^-$ is supported by $A^-$. It is enough to take $A=A^+\setminus(-A^+)$.

To show (b) it is enough to observe that $x\mapsto \int y^*(x)\di\mu(y^*)$ defines a linear functional on $E$ of norm not greater than $\|\mu\|$.

(c) If $\mu=0$, take $\nu=\frac12(\ep_{y^*}+\ep_{-y^*})$ where $y^*$ is an extreme point of $X$.
If $\mu\ne0$, one can take $\nu=2\mu^+ +\frac{1-2\norm{\mu^+}}{\norm{\mu}}\abs{\mu}$. Indeed, let $A$ be the set provided by (a). Then $\norm{\mu^+}=\mu^+(A)=\mu(A)=-\mu(-A)=\mu^-(-A)=\norm{\mu^-}$, thus $\norm{\mu^+}\le\frac12$. It follows that $\nu$ is a probability measure.
Moreover, given $B\subset X$ Borel, we have
$$\begin{aligned}
\odd\nu(B)&=\frac12(\nu(B)-\nu(-B))
=\mu^+(B)-\mu^+(-B)+\tfrac{1-2\norm{\mu^+}}{2\norm{\mu}} (\abs{\mu}(B)-\abs{\mu}(-B))
\\&=\mu(B\cap A)-\mu((-B)\cap A)\\&\quad+ \tfrac{1-2\norm{\mu^+}}{2\norm{\mu}} (\mu(B\cap A)-\mu(B\cap(-A))-\mu((-B)\cap A)+\mu((-B)\cap(-A)))
\\&=\mu(B\cap A)+\mu(B\cap(-A)) + 0 = \mu(B),\end{aligned}$$
thus $\odd\nu=\mu$.

Let $f:X\to\er$ be a continuous affine function. Then $f-f(0)$ is odd, thus there is $x\in E$ such that $f(x^*)-f(0)=x^*(x)$ for each $x^*\in X$. Hence
$$\nu(f)=\nu(f(0))+\nu(f-f(0))=f(0)+\mu(f-f(0))=f(0)+x^*(x)=f(x^*),$$
therefore $\nu$ represents $x^*$.

Finally, if $\mu$ is boundary, then both $|\mu|$ and $\mu^+$ are boundary, hence $\nu$ is maximal.
\end{proof}

The next lemma sums up the properties of the operator $\hom$ on complex measures. It is analogous to Lemma~\ref{l:odd-miry} but the proof is more involved due to the more complicated definition in the complex case.

\begin{lemma}\label{l:hom-miry1}
Let $E$ be a complex Banach space, $X=(B_{E^*},w^*)$ and $\mu$ be a complex Radon measure on $X$.
Then the following assertions hold.
\begin{itemize}
	\item[(a)] $\hom\mu(A)=\frac1{2\pi}\int_0^{2\pi}e^{-it}\mu(e^{-it}A)\di t$ for any  Borel set $A\subset X$.
	\item[(b)] $\hom\mu$ is an anti-homogeneous measure.
	\item[(c)] $\mu$ is anti-homogeneous if and only if $\mu=\hom\mu$.
	\item[(d)] Let $B\subset X$ be a homogeneous $\mu$-measurable set and $f:B\to \ce$ a bounded Borel function. Then $\int_B f\di\hom\mu=\int_B\hom f\di\mu$.
 \item[(e)] If $\mu$ is boundary, then $\hom\mu$ is boundary.
\end{itemize}
\end{lemma}

\begin{proof}
(a) Let $\lambda$ denote the normalized Lebesgue measure on $[0,2\pi]$ and  $\lambda\times \mu$ the product measure on $[0,2\pi]\times X$. Set $h(t,x)=e^{-it}$. Then $h$ is a bounded continuous function on $[0,2\pi]\times X$, hence we can set $\nu_1=h\cdot(\lambda\times\mu)$ to be the measure on $[0,2\pi]\times X$ with the density $h$ with respect to $\lambda\times\mu$. Further, consider the mapping $\varphi:[0,2\pi]\times X\to X$ defined by $\varphi(t,x)=e^{it}x$. This mapping is clearly continuous, thus the image of the product measure $\nu_2=\varphi(\nu_1)$ is a well-defined Radon measure on $X$. Then for any $A\subset X$ Borel we have
$$\begin{aligned}
\nu_2(A)&=\nu_1(\varphi^{-1}(A))=\nu_1(\{(t,x)\setsep e^{it}x\in A\})=
\nu_1(\{(t,e^{-it}x)\setsep t\in[0,2\pi],x\in A\})
\\&=\int_{\{(t,e^{-it}x)\setsep t\in[0,2\pi],x\in A\}} e^{-it}\di(\lambda\times\mu)
=\frac1{2\pi}\int_0^{2\pi}e^{-it}\mu(e^{-it}A)\di t.
\end{aligned}$$
This shows that the formula in the statement of (a) defines the measure $\nu_2$. To show $\nu_2=\hom\mu$ it is enough to prove that $\nu_2(f)=\mu(\hom f)$ for any $f\in\C(X,\ce)$. So, fix $f\in\C(X,\ce)$. Then
$$\begin{aligned}
\mu(\hom f)
&=\int_X\left(\frac1{2\pi}\int_0^{2\pi}e^{-it}f(e^{it} x)\di t\right)\di\mu
=\int_{[0,2\pi]\times X} e^{-it}f(e^{it} x) \di(\lambda\times\mu)(t,x)
\\&=\int_{[0,2\pi]\times X}f(e^{it} x) \di\nu_1(t,x)
=\int_{[0,2\pi]\times X}f(\varphi(t,x)) \di\nu_1(t,x)
=\int_X f\di\nu_2,
\end{aligned}$$
which completes the proof.

The assertions (b) and (c) follow immediately from the formula in (a).
The assertion (d) follows from the formula in (a) since the computation of $\mu(\hom f)$ can be repeated for any bounded Borel function $f$.

(e) For the proof see \cite[Lemma 4.2]{effros} or \cite[\S\,23, Lemma 10]{lacey}.
\end{proof}

We continue with the complex analogue of Lemma~\ref{l:odd-miry2}. The proof is again more involved than in the real case.

\begin{lemma}
\label{l:hom-miry2}
Let $E$ be a complex Banach space, $X=(B_{E^*},w^*)$ and $\mu$ be an anti-homogeneous Radon measure on $X$ with $\|\mu\|\le 1$.
Then the following assertions hold.
\begin{itemize}
	\item[(a)] $\Re\mu$ and $\Im\mu$ are odd measures.
	\item[(b)] $\Im\mu (A)=\Re\mu(iA)$ for $A\subset X$ Borel.
	\item[(c)] There is a unique $x^*\in X$ such that for each $x\in E$ one has $x^*(x)=\int y^*(x)\di\mu(y^*)$.
	\item[(d)] There is a Radon probability $\nu$ on $X$ such that $\hom\nu=\mu$ and $r(\nu)=x^*$. If $\mu$ is boundary, $\nu$ can be chosen maximal.
\end{itemize}
\end{lemma}

\begin{proof} The assertions (a) and (b) are obvious.
 To show (c) it is enough to observe that $x\mapsto \int y^*(x)\di\mu(y^*)$
 defines a linear functional on $E$ of norm at most equal to $\|\mu\|$.

(d) If $\mu=0$ take $\nu=\frac12(\ep_{y^*}+\ep_{-y^*})$ where $y^*$ is an extreme point of $X$.

Let $\mu\ne0$. Then the absolute variation $\abs{\mu}$ is invariant with respect to rotations, i.e., $\abs{\mu}(\alpha A)=\abs{\mu}(A)$ for any $A\subset X$ Borel and $\alpha\in\ce$ with $\abs{\alpha}=1$. Let $h_0$ be a Borel function on $X$ such that $\mu=h_0\abs{\mu}$ (i.e., $h_0$ is the density of $\mu$ with respect to $\abs{\mu}$). Then $\abs{h_0}=1$ $\abs{\mu}$-almost everywhere. Observe that without loss of generality we may suppose that $h_0$ is Baire measurable. Indeed, by the Luzin theorem there is a sequence $(f_n)$ of continuous functions converging $\abs{\mu}$-almost everywhere to $h_0$. Then the function $h_1$ defined by
$$h_1(y^*)=\begin{cases}\lim f_n(y^*) &\mbox{if the limit exists},\\ 0  &\mbox{otherwise},\end{cases}$$
is a Baire measurable function which equals $h_0$ $\abs{\mu}$-almost everywhere. Further, define the function $h$ by the formula
$$
h(y^*)=\frac1{2\pi}\int_0^{2\pi}e^{-it}h_1(e^{-it}y^*)\di t,\quad y^*\in X.
$$
Then $h$ is also Baire measurable (by Lemma~\ref{L:baire} and Lemma~\ref{l:hom}(c)). Moreover, $h(\alpha y^*)=\ov{\alpha}h(y^*)$ for any $y^*\in X$ and any complex unit $\alpha$. Finally, $h=h_1$ $\abs{\mu}$-almost everywhere since $h$ is also a density of $\mu$ with respect to $\abs{\mu}$. Indeed, for any Borel set $A\subset X$ we have
$$\begin{aligned}
\mu(A)&=\hom\mu(A)=\frac1{2\pi}\int_0^{2\pi}e^{-it}\mu(e^{-it}A) \di t
\\&=\frac1{2\pi}\int_0^{2\pi}e^{-it}\int_{e^{-it}A} h_1(y^*)\di\abs{\mu}(y^*) \di t
\\&=\frac1{2\pi}\int_0^{2\pi}e^{-it}\int_{A} h_1(e^{-it}y^*)\di\abs{\mu}(y^*) \di t
\\&=\int_{A}\frac1{2\pi}\int_0^{2\pi}e^{-it}h_1(e^{-it}y^*) \di t\di\abs{\mu}(y^*)
=\int_A h(y^*)\di\abs{\mu}(y^*).
\end{aligned}$$
Set $X_1=\{y^*\in X\setsep \abs{h(y^*)}=1\}$ and $X_0=X\setminus X_1$. Then $X_0$ and $X_1$ are Baire subsets of $X$, $0\in X_0$ (since $h(0)=0$) and $X_0$ has $\abs{\mu}$-measure zero. Further, set $P=\{y^*\in X\setsep h(y^*)=1\}$. Then $P$ is a Baire subset of $X$ and the mapping $\Phi:[0,2\pi)\times P\to X_1$ defined by
$\Phi(t,y^*)=e^{it}y^*$ is a continuous bijection. Since the product space $Z=[0,2\pi)\times P$ is a Baire subset of the compact space $[0,2\pi]\times X$, $\Phi$ maps Baire sets in $Z$ to Baire sets of $X_1$. (Indeed, let $A\subset Z$ be a Baire set. Then both $A$ and $Z\setminus A$ are $K$-analytic,
hence $\Phi(A)$ and $\Phi(Z\setminus A)$ are two disjoint $K$-analytic subsets of $X_1$ covering $X_1$, hence they are Baire sets, see \cite[Theorems 4.10 and 5.8]{frolikcmj}.)

Therefore we can define measure $\nu_1$ as $\nu_1=\Phi^{-1}(\abs{\mu})$. Then $\nu_1$ is a positive measure defined on the Baire $\sigma$-algebra of $Z$. Further, define $\nu_2=\pi_2(\nu_1)$, where $\pi_2$ is the projection of $Z=[0,2\pi)\times P$ onto the second coordinate. Then $\nu_2$ is a positive measure defined on the Baire $\sigma$-algebra of $P$. Since $P$ is a Baire subset of $X$, we can consider $\nu_2$ to be defined on the Baire $\sigma$-algebra of $X$ (and supported by $P$). Therefore $\nu_2$ generates a positive functional on $\C(X,\ce)$. Since
$$\norm{\nu_2}=\nu_2(1)=\nu_2(P)=\nu_1(Z)=\abs{\mu}(X_1)=\norm{\mu},$$
we have $\norm{\nu_2}\le 1$. Moreover, if $f\in\C(X,\ce)$ is homogeneous, then
$$\begin{aligned}
\int_X f(y^*)\di\nu_2(y^*)&=\int_P f(y^*)\di\nu_2(y^*)
=\int_Z f(y^*)\di\nu_1(t,y^*)
=\int_Z e^{-it}f(e^{it}y^*)\di\nu_1(t,y^*)
\\&=\int_Z h(e^{it}y^*) f(e^{it}y^*)\di\nu_1(t,y^*)
=\int_{X_1} h(z^*)f(z^*)\di\abs{\mu}(z^*)
\\&=\int_{X} h(z^*)f(z^*)\di\abs{\mu}(z^*)
=\int_{X}f(z^*)\di{\mu}(z^*).
\end{aligned}$$
The first equality holds because $\nu_2$ is supported by $P$, the second one follows from the fact that $\nu_2=\pi_2(\nu_1)$. The third is valid since $f$ is homogeneous. In the fourth one we used the properties of $h$: $h(e^{it}y^*)=e^{-it}h(y^*)=e^{-it}$ for $y^*\in P$. The fifth inequality follows from the fact that $\nu_1=\Phi^{-1}(\abs{\mu})$, in the sixth one we used that $\abs{\mu}$ is supported by $X_1$ and, finally, the last one follows from the choice of $h$. Hence, for any $f\in\C(X,\ce)$ we have
$$\hom\nu_2(f)=\nu_2(\hom f)=\mu(\hom f)=\mu(f),$$
hence $\hom\nu_2=\mu$.

Set $\nu=\nu_2+ \frac{1-\norm{\nu_2}}{\norm{\mu}}\abs{\mu}$. Then $\nu$ is a probability measure and $\hom\nu=\mu$ (since $\hom\abs{\mu}=0$).
We continue by showing that $\nu$ represents $x^*$. Let $f$ be an affine continuous function on $X$. By Lemma~\ref{l:homaff}, there are $u,v$ homogeneous affine continuous functions from $X$ to $\ce$ such that $f=f(0)+u+\ov{v}$. Since $u$ and $v$ are given by evaluation at some points of $E$, we get
$$\nu(f)=\nu(f(0)+u+\ov{v})=f(0)+\nu(u)+\ov{\nu(v)}=f(0)+u(x^*)+\ov{v(x^*)}=f(x^*).$$

It remains to show that $\nu$ is maximal whenever $\mu$ is boundary. Since $\nu$ is a probability and $\abs{\mu}$ is a boundary measure, it is enough to check that $\nu_2$ is boundary. We will do that by the test using convex continuous functions. For a real-valued continuous convex function  $f$ on $X$ set
$$\hat{f}(y^*)=\sup\{\sigma(f)\setsep\sigma\in\M^1(X), r(\sigma)=y^*\}, \quad y^*\in X,$$ and
$$B(f)=\{y^*\in X\setsep f(y^*)=\hat{f}(y^*)\}.$$ Then boundary measures are exactly the measures supported by each $B(f)$ (see \cite[p. 34--35]{alfsen}). So, let $f$ be any real-valued continuous convex function on $X$. Set $$g(y^*)=\frac1{2\pi}\int_0^{2\pi} f(e^{it}y^*)\di t, \quad y^*\in X.$$
 Then $g$ is again a continuous convex function and, moreover, $B(g)\subset B(f)$ by \cite[Lemma 4.1]{effros}. Hence $\abs{\mu}(X\setminus B(g))=0$. Since $X\setminus B(g)$ is homogeneous, by the definition of $\nu_2$ we get $\nu_2(X\setminus B(g))=\abs{\mu}(X\setminus B(g))=0$. Hence $\nu_2$ is supported by $B(g)$ and, a fortiori, by $B(f)$. This completes the proof.
 \end{proof}

Now we are ready to provide the proof of Fact~\ref{fact:L1}:

\begin{proof}[Proof of Fact~\ref{fact:L1}.] \textit{{The real case:}} By \cite[\S 21, Theorem 7]{lacey}, $E$ is an
$L_1$-predual if and only if for any two maximal probability measures $\mu$,$\nu$ on $X$ with the same barycenter we have $\odd\mu=\odd\nu$.

Thus the sufficiency is clear: If $\mu$ and $\nu$ are two maximal probability measures with the same barycenter $x^*$, then $\odd\mu$ and $\odd\nu$ satisfy the conditions (a)--(d) (by Lemma~\ref{l:odd-miry}). So, by the uniqueness assumption we have $\odd\mu=\odd\nu$.

To show the necessity suppose that $E$ is an $L^1$-predual. Let $x^*\in X$ be arbitrary. Then there is a maximal probability $\nu$ representing $x^*$.
Set $\mu=\odd\nu$. Then $\mu$ satisfies (a)--(d) by Lemma~\ref{l:odd-miry}. To prove the uniqueness suppose that $\mu_1$ is any measure satisfying (a)--(d). By Lemma~\ref{l:odd-miry2} there is a maximal probability measure $\nu_1$ with  the barycenter $x^*$ such that $\odd\nu_1=\mu_1$.
Thus $\mu=\mu_1$.	

\textit{The complex case} is completely analogous. By \cite[Theorem 4.3]{effros} or \cite[\S 23, Theorem 5]{lacey}, $E$ is an $L_1$-predual if and only if for any two maximal probability measures $\mu$,$\nu$ on $X$ with the same barycenter we have $\hom\mu=\hom\nu$. Thus the sufficiency is again clear and the necessity can be proved in the same way, only using Lemma~\ref{l:hom-miry2}.

\end{proof}

\section{Affine maps of the first class and approximation properties}\label{sec:example}

The aim of this section is to prove Theorem~\ref{T:bap} and Example~\ref{E:cap}. To do that we first give the following observation on strongly affine maps.

\begin{lemma}\label{4.3}
Let $f\colon K\to F$ be a strongly affine mapping from a compact convex set $K$ to a locally convex space $F$ over $\ef$.
Then $f$ is bounded.\end{lemma}

\begin{proof} Let $\tau\in F^*$ be arbitrary. By Fact~\ref{fact1} the mapping $\tau\circ f$ is strongly affine. By  \cite[Theorem~I.2.6]{alfsen} (or \cite[Section~14]{phelps-choquet}, \cite{sray-om} or \cite[Lemma 4.5]{lmns}) we get that $\tau\circ f$ is bounded. Hence $f$ is bounded by \cite[Theorem 4 on p. 151]{jarchow}.
\end{proof}

We continue by recalling definitions of some approximation properties of Banach spaces.

A Banach space $E$ is said to have
\begin{itemize}
  \item the \emph{approximation property} if for every $\ep>0$ and every compact set $K\subset E$ there exists a finite-rank operator $L$ on $E$ such that $\sup_{x\in K}\norm{Lx-x}\le \ep$.
	\item the \emph{compact approximation property} if for every $\ep>0$ and every compact set $K\subset E$ there is a compact operator $L$ on $E$ such that $\sup_{x\in K}\norm{Lx-x}\le \ep$.
\end{itemize}

In other words, $E$ has the (compact) approximation property if and only if the identity operator on $E$ is in the closure of finite-rank (compact) operators in the topology of uniform convergence on compact subsets of $E$. If the identity can be approximated by the operators of the respective type with norm at most $\lambda$, $E$ is said to have the \emph{$\lambda$-bounded approximation property} (\emph{$\lambda$-bounded compact approximation property}). Further, $E$ is said to have the \emph{bounded approximation property} if it has
the \emph{$\lambda$-bounded approximation property} for some $\lambda\ge1$.

Now we give the proof of Example~\ref{E:cap}. It is a strengthening of \cite[Example 2.22]{MeSta} which follows from \cite[Proposition 2.21]{MeSta}. In the quoted paper the authors prove that $f\notin\fra_1(X,E)$, we prove that it is
not in any affine Baire class.

\begin{proof}[Proof of Example~\ref{E:cap}]
Since $E$ is reflexive, the unit ball $B_E$ is weakly compact. By the Namioka theorem (see e.g. \cite[Corollary 14.4]{hhz}), the function $f$ (which is the identity from $X=(B_E,w)$ to $E$ equipped with the norm) has a dense (in fact residual) set of continuity points. The same is true for the restriction of $f$ to any closed subset of $X$. Since $E$ is separable, $X$ is moreover metrizable, thus $f$ is $F_\sigma$-measurable by \cite[\S31, X, Theorem 2]{kuratowski}. This implies that $f\in \C_1(X,B_E)$ by Lemma~\ref{L:baire}(c).

Suppose that $f\in\bigcup_{\alpha<\omega_1}\fra_\alpha(X,E)$. Since $f$ is odd (homogeneous), we get $f\in \bigcup_{\alpha<\omega_1}\fra_{\odd,\alpha}(X,E)$ by Lemma~\ref{l:odd} ($f\in \bigcup_{\alpha<\omega_1}\fra_{\hom,\alpha}(X,E)$ by Lemma~\ref{l:hom}).
Let $I$ denote the identity operator on $E$. Using Lemma~\ref{l:operator} we then get
$I\in\bigcup_{\alpha<\omega_1}(K(E))_\alpha$, where the notation $(K(E))_\alpha$ follows the pattern from the Section~\ref{ssec:baire}. (Indeed, a linear operator $L:E\to E$ is weak-to-norm continuous on $B_E$ if and only if it is compact.)
To finish the argument it is enough to check that all the operators from $\bigcup_{\alpha<\omega_1}(K(E))_\alpha$ are bounded and are contained in the closure of $K(E)$ in the topology of uniform convergence on norm compact subsets of $E$. This will be done by transfinite induction. For $\alpha=0$ it is clear. Let $\alpha>0$ and suppose that the assertion is valid for each operator in $\bigcup_{\beta<\alpha}(K(E))_\beta$. Fix $L\in (K(E))_\alpha$. Then there is
a sequence $(L_n)$ in $\bigcup_{\beta<\alpha}(K(E))_\beta$ pointwise converging to $L$. By the induction hypothesis the operators $L_n$ are bounded,
hence the uniform boundedness principle shows that the sequence $(L_n)$ is uniformly bounded. Thus $L$ is bounded and, moreover, $L_n\to L$ uniformly on compact sets.
\end{proof}

To prove Theorem~\ref{T:bap} we need some lemmata. Let us first introduce some notation.

Let $E,F$ be Banach spaces such that $F$ has dimension $n\in\en$. Fix a basis $y_1,\dots,y_n$ of $F$ and denote by $y_1^*,\dots,y_n^*$ the dual basis of $F^*$. For an $n$-tuple $\boldsymbol x^{**}=(x^{**}_1,\dots,x^{**}_n)\in (E^{**})^n$ define operators $\Phi(\boldsymbol x^{**})\in L(E^*,F)$ and $\Psi(\boldsymbol x^{**})\in L(F^*,E^{**})$ by the formulas
$$\begin{gathered}
\Phi(\boldsymbol x^{**})(x^*)=\sum_{i=1}^n x^{**}_i(x^*)y_i,\quad x^*\in E^*,\\
\Psi(\boldsymbol x^{**})(y^*)=\sum_{i=1}^n y^*(y_i)x^{**}_i,\quad x^*\in F^*.
\end{gathered}$$
It is clear that $\Phi$ is an isomorphism of $(E^{**})^n$ onto $L(E^*,F)$ and $\Psi$ is an isomorphism of  $(E^{**})^n$ onto $L(F^*,E^{**})$.

\begin{lemma}\label{L:locref} Under the above notation the following assertions hold.
\begin{itemize}
	\item[(i)] $\norm{\Phi(\boldsymbol x^{**})}=\norm{\Psi(\boldsymbol x^{**})}$ for each $\boldsymbol x^{**}\in (E^{**})^n$.
	\item[(ii)] $\Phi$ restricted to $E^n$ is an isomorphism of $E^n$ onto the subspace of $L(E^*,F)$ formed by weak$^*$ continuous operators.
	\item[(iii)] $\Psi$ restricted to $E^n$ is an isomorphism of $E^n$ onto $L(F^*,E)$.
	\item[(iv)] If we define on $E^n$ a norm $\norm{\cdot}$ by the formula $\norm{\boldsymbol x}=\norm{\Phi(\boldsymbol x)}$, then the bidual norm on $(E^{**})^n$ coming from the canonical duality of $E^n$, $(E^*)^n$ and $(E^{**})^n$ is given by the formula
	$\norm{\boldsymbol x^{**}}=\norm{\Phi(\boldsymbol x^{**})}$.
\end{itemize}
\end{lemma}

\begin{proof} (i) Let $\boldsymbol x^{**}\in(E^{**})^n$. Then
$$\begin{aligned}
\norm{\Phi(\boldsymbol x^{**})}
&=\sup\{ \norm{\Phi(\boldsymbol x^{**})(x^*)}_F \setsep x^*\in B_{E^*} \}
\\&=\sup\{ \abs{y^*(\Phi(\boldsymbol x^{**})(x^*))} \setsep x^*\in B_{E^*},y^*\in B_{F^*} \}
\\&=\sup\left\{ \abs{\sum_{i=1}^n x^{**}_i(x^*)y^*(y_i)} \setsep x^*\in B_{E^*},y^*\in B_{F^*} \right\}
\\&=\sup\{ \abs{\Psi(\boldsymbol x^{**})(y^*)(x^*)} \setsep x^*\in B_{E^*},y^*\in B_{F^*} \}
\\&=\sup\{ \norm{\Psi(\boldsymbol x^{**})(y^*)}_{E^{**}} \setsep y^*\in B_{F^*} \}
=\norm{\Psi(\boldsymbol x^{**})}.
\end{aligned}$$

The assertion (ii) follows from the fact that the weak$^*$ continuous functionals on $E^*$ are exactly those which come from $E$. The assertion (iii) is obvious.

(iv) Due to (i) it is enough to prove this assertion with $\Psi$ instead of $\Phi$. But this one is proved for example
in \cite[Section 3]{dean}.
\end{proof}

\begin{lemma}\label{L:mok-fin}
Let $X$ be a compact convex set, $F$ be a finite-dimensional Banach space and $f:X\to F$ be an affine function which belongs to $\C_1(X,F)$ and satisfies $\norm{f(x)}_F\le1$ for each $x\in X$. Then $f\in\fra_1(X,B_F)$.
\end{lemma}

\begin{proof} Without loss of generality we may suppose that $F$ is a real space. Since $F$ has finite dimension, it easily follows from the Mokobodzki theorem that $f\in\fra_1(X,F)$. Let $E=\fra(X)$. For any $x\in X$ let $\ep_x\in E^*$ be the respective evaluation functional. By \cite[Lemma 2.2]{affperf} there is a unique linear operator $L_f:E^*\to F$ such that
$L_f(\ep_x)=f(x)$ for each $x\in X$. If follows from \cite[Lemma 2.3]{affperf} that $\norm{L_f}\le 1$ and $L_f\r_{B_{E^*}}\in\fra_1((B_{E^*},w^*),F)$. The Banach-Dieudonn\'e theorem easily yields that $L_f\in\fra_1((E^*,w^*),F)$.

Let $H$ denote the subspace of $L(E^*,F)$ formed by the weak$^*$ continuous operators. By Lemma~\ref{L:locref}(iv) the bidual $H^{**}$ can be canonically identified with $L(E^*,F)$. Hence, in this identification we have $L_f\in H^{**}$ and, moreover,
$L_f$ is the weak$^*$-limit of a sequence from $H$. Since $\|L_f\|\le1$, it follows from \cite[Remark on p. 379]{odro} that $L_f$ is the weak$^*$-limit of a sequence from $B_H$. But this means that $L_f\in\fra_1((E^*,w^*),B_F)$, hence $f\in\fra_1(X,B_F)$.
\end{proof}

\begin{proof}[Proof of Theorem~\ref{T:bap}] It is enough to prove the `moreover' part since any affine mapping $f\in\C_1(X,E)$ is bounded by Theorem~\ref{T:c1sa} and Lemma~\ref{4.3}.

 Suppose that $E$ has the $\lambda$-bounded approximation property for some $\lambda\ge 1$. Let $f:X\to E$ be an affine mapping of the first Baire class such that $f(X)\subset B_E$. Observe that $f(X)$ is separable and fix a countable dense subset $C\subset f(X)$. The  $\lambda$-bounded approximation property yields a sequence $(L_n)$ of finite-rank operators pointwise converging to the identity on $C$ such that $\norm{L_n}\le\lambda$ for each $n$. It follows that $L_n$ converge to the identity also pointwise on $f(X)$.  Then $L_n\circ f$ is also affine and of the first Baire class. Moreover, since the range has finite dimension and $(L_n\circ f)(X)\subset\lambda B_E$, Lemma~\ref{L:mok-fin} yields $L_n\circ f\in\fra_1(X,\lambda B_E)$. Finally, $L_n\circ f\to f$ pointwise, we get $f\in\fra_2(X,\lambda B_E)$. Finally, Corollary~\ref{C:c21} yields $f\in\fra_1(X,\lambda B_E)$.
\end{proof}

\begin{remark} Similar problems are investigated in \cite{MeSta}. In particular, Theorem 2.12 in the quoted paper is a weaker variant of our Theorem~\ref{T:bap}. However, the proof given in \cite{MeSta} contains a gap. The authors confuse a projection on a one-dimensional space with a coordinate functional. Therefore, it is not clear why the family of operators $(T_{n,m})$ defined on p.\ 25 should be uniformly bounded. The correct estimate is $\norm{T_{n,m}}\le 2Kn\norm{T}$ which is not a uniform bound. The same type of a gap is in Remark 2.18.1 and in the proof of Theorem 2.19. These results can be repaired and improved using our Lemma~\ref{L:mok-fin}.
\end{remark}
%%%%%%%%%%%%%%%%%%%%%%%%%%%%%%%%%%%%%%%%%%%%%%%%%%%
%%%%%%%%%%%%%%%%%%%%%%%%%%%%%%%%%%%%%%%%%%%%%%%%%%%%

\section{The affine class of the dilation mapping}\label{Sec:dilation}

The aim of this section is to prove Theorem~\ref{T:dilation} and Remark~\ref{rem:dilation}.
We start by proving Remark~\ref{rem:dilation}. First we prove that the dilation mapping is always strongly affine.
It is the content of the following lemma.

\begin{lemma}\label{l:te-sa} \
\begin{itemize}
	\item[(S)] Let $X$ be a simplex. Then the mapping $T:x\mapsto\delta_x$ is strongly affine.
	\item[(R,C)] Let $X=(B_{E^*},w^*)$ for an $L^1$-predual $E$. Then $T$ is strongly affine.
\end{itemize}
\end{lemma}

\begin{proof}
(S) It is clear that the mapping is affine. (Indeed, if $x,y\in X$ and $t\in[0,1]$, then $t\delta_x+(1-t)\delta_y$ is a probability measure with barycenter $tx+(1-t)y$. Since maximal probabilities form a convex set (by \cite[Theorem 3.70]{lmns}), necessarily $t\delta_x+(1-t)\delta_y=\delta_{tx+(1-t)y}$.) To prove that it is strongly affine, by Fact~\ref{fact1} it is enough to prove that $x\mapsto\delta_x(f)$ is strongly affine for each $f:X\to\er$ continuous. Since differences of convex continuous functions are norm-dense in $\C(X,\er)$ by the Stone-Weierstrass theorem, it is enough to prove it for convex continuous functions. But if $f$ is convex and continuous, then the mapping $x\mapsto \delta_x(f)$ is upper semicontinuous by \cite[Theorem II.3.7]{alfsen}, hence it is strongly affine by \cite[Proposition A.122 and Theorem 4.21]{lmns}.

(R) The proof in this case is completely analogous to the case (S), we only need to use \cite[Chapter 7, \S 21, Theorem 7]{lacey}
to show that $Tf$ is a difference of two upper semicontinuous functions whenever $f$ is continuous and convex.

(C) The assertion for complex $L^1$-preduals is proved in \cite[Lemma~4.12]{petracek-spurny}.
\end{proof}

\begin{lemma}\label{l:bauer}
If $\ext X$ is closed, then $T$ is continuous.
\end{lemma}

\begin{proof}
(S) The mapping $r:\M^1(X)\to X$ which assigns to each $\mu\in\M^1(X)$ its barycenter $r(\mu)$ is clearly continuous (see, e.g., \cite[Proposition 3.40]{lmns}).
Since $\ext X$ is compact, $\M^1(\ext X)$ is a compact subset of $\M^1(X)$. Moreover, maximal probabilities are exactly those supported by $\ext X$ (see, e.g., \cite[p. 35]{alfsen} or \cite[Proposition 3.80]{lmns}). By the simpliciality the operator $r$ restricted to $\M^1(\ext X)$ is one-to-one, hence it is a homeomorphism. The operator $T$ is its inverse.

(R) The proof is similar. Let $Y$ denote the set of all odd measures from $B_{\M(X,\er)}$ and  $Y_b$ the set of all boundary measures from $Y$. For any $\mu\in Y$ let $u(\mu)$ denotes the point of $X$ provided by Lemma~\ref{l:odd-miry2}(b). It is clear that the operator $u$ is continuous and, by Fact~\ref{fact:L1}, the restriction of $u$ to $Y_b$ is one-to-one. The operator $T$ is the inverse of $u\r_{Y_b}$. To prove that $T$ is continuous it is enough to check that $Y_b$ is compact. But $Y_b$ is the image of the set of all maximal probabilities on $X$ under the continuous operator $\odd$ (by Lemma~\ref{l:odd-miry2}(c)) and maximal probabilities on $X$ are exactly the probabilities from the compact set $\M^1(\ext X)$ (by  \cite[Proposition 3.80]{lmns}).

(C) The complex case is completely analogous to the real one. We just consider anti-homogeneous measures instead of odd ones and use Lemma~\ref{l:hom-miry2} instead of Lemma~\ref{l:odd-miry2}.
\end{proof}

Now we proceed with the proof of Theorem~\ref{T:dilation}. It is based on the following selection result, which can be viewed as an affine version of the Michael selection theorem. We recall that a set-valued mapping $\Phi\colon X\to F$ is said to be
\emph{lower semicontinuous} if $\{x\in X\setsep \Phi(x)\cap U\ne\emptyset\}$ is open in $X$ for any open set $U\subset F$.

\begin{thm}\label{T:selekcelsc}
Let $X$ be a compact convex set, $F$ a \fr space and $\Phi\colon X\to F$ a lower semicontinuous set-valued mapping with nonempty closed values.
\begin{itemize}
	\item[(S)] If $X$ is a simplex  and the graph of $\Phi$ is convex, $\Phi$ admits a continuous affine selection.
	\item[(R)] If $X=(B_{E^*},w^*)$, where $E$ is a real $L_1$-predual and the graph of $\Phi$ is convex and symmetric, $\Phi$ admits an odd continuous affine selection.
	\item[(C)] If $X=(B_{E^*},w^*)$, where $E$ is a complex $L_1$-predual, $F$ is complex and the graph of $\Phi$ is absolutely convex, $\Phi$ admits a homogeneous continuous affine selection.
\end{itemize}
\end{thm}

The case (S) is due to \cite[Theorem 3.1]{lazar-sel} (see also \cite[Theorem 11.6]{lmns}), the case (R) is proved in \cite[Theorem 2.2]{lali} or \cite[Chapter 7, \S 22, Theorem 2]{lacey} and the case (C) in \cite[Theorem 4.2]{olsen-sel}. More precisely, the assumptions of the quoted results are formulated
in a bit different way. Instead of assuming that the graph of $\Phi$ is convex, it is assumed that $\Phi$ is affine (sometimes called convex), i.e.,
$\lambda\Phi(x_1)+(1-\lambda)\Phi(x_2)\subset\Phi(\lambda x_1+(1-\lambda)x_2)$ whenever $x_1,x_2\in X$ and $\lambda\in[0,1]$. But these two assumptions are indeed equivalent.

\begin{proof}[Proof of Theorem~\ref{T:dilation}]
(S) This case is essentially due to \cite[Theorem 6.6]{lmnss03} (or \cite[Theorem 11.26]{lmns}). However, the formulation of these results is weaker and hence we give a complete proof here.
Since $X$ is metrizable, $\fra(X,\er)$ is separable, hence we can choose a countable set $\{e_n\setsep n\in\en\}$ dense in $\fra(X,\er)$.

Let $n\in\en$ be fixed. We define a multivalued mapping $\Gamma_n\colon X\to \M^1(X)$ by the formula
\[
\Gamma_n(x)=\{\mu\in \M^1(X) \setsep \abs{\mu(e_i)-e_i(x)}<\frac1n, i=1,\dots, n\},\quad x\in  X.
\]
We claim that $\Gamma_n$ is a nonempty valued lower semicontinuous mapping with convex graph. First, $\Gamma_n(x)\ne\emptyset$ as $\ep_x\in\Gamma_n(x)$. Further, the graph of $\Gamma_n$ is equal to
$$\{(x,\mu)\in X\times \M^1(X)\setsep  \abs{\mu(e_i)-e_i(x)}<\frac1n, i=1,\dots, n\},$$
hence it is convex.

We continue by showing that $\Gamma_n$ is lower semicontinuous.
Let $V\subset \M^1(X)$ be a nonempty open set and let $\Gamma_n(x)\cap V\neq \emptyset$ for some $x\in X$.
We select $\mu\in V$ satisfying $\abs{\mu(e_i)-e_i(x)}<\frac1n$ for each $i=1,\dots, n$. By the continuity of the functions $e_i$, $i=1,\dots,n$, there exists an open  neighborhood $U$ of $x$ satisfying
\[
\abs{\mu(e_i)-e_i(y)}<\frac1n,\quad y\in U, i=1,\dots,n.
\]
Hence $\mu\in \Gamma_n(y)\cap V$ for $y\in U$, and thus the set
\[
\{x\in X\setsep \Gamma_n(x)\cap V\neq\emptyset\}
\]
is open. This implies that $\Gamma_n$ is lower semicontinuous.

Next we modify the mapping $\Gamma_n$ to have closed values. To this end
we define a mapping $\ov{\Gamma}_n\colon X\to \M^1(X)$ as
\[
\ov{\Gamma}_n(x)=\ov{\Gamma_n(x)},\quad x\in X.
\]
Since, for an open set $V\subset \M^1(X)$ and $x\in X$, $\ov{\Gamma}_n(x)$ intersects $V$ if and only if $\Gamma_n(x)$ intersects $V$, $\ov{\Gamma}_n$ is also lower semicontinuous. Moreover, it is clear that the graph of $\ov{\Gamma}_n$ convex.

Now we want to use the selection result contained in Theorem~\ref{T:selekcelsc}. This is possible
since $\M^1(X)$, being a compact convex metrizable set, is affinely homeomorphic to a subset of $\ell_2$.
Hence there exists a continuous  affine selection $T_n$ from $\ov{\Gamma}_n$.

We continue by showing that $T_n\to T$ on $\ext X$. Fix $x\in \ext X$. Let $\mu$ be any cluster point of the sequence $(T_n(x))$ in $\M^1(X)$. By the very definition of $\Gamma_n$ and $\ov{\Gamma}_n$, $\mu(e_i)=e_i(x)$ for $i\in\en$. By the density of the sequence $(e_i)$ in $\fra(X,\er)$ we get that
$\mu(e)=e(x)$ for all $e\in \fra(X,\er)$, i.e., $r(\mu)=x$. Since $x$ is an extreme point of $X$, necessarily $\mu=\ep_{x}=\delta_x$. It follows that $T_n(x)\to \delta_x$.

Finally, we will prove that $T_n\to T$ on $X$. Fix any $x\in X$ and $f\in C(X,\er)$.
Then
$$\delta_x(f)=\int_X \delta_y(f)\di\delta_x(y)=\lim_{n\to\infty} \int_X T_n(y)(f)\di\delta_x(y)=\lim_{n\to\infty} T_n(x)(f).$$
The first equality follows by the strong affinity of $T$ (see Lemma~\ref{l:te-sa}). To verify the second one we use that $\delta_x$ is maximal, hence supported by $\ext X$ ,
the already proved fact that $T_n\to T$ on $\ext X$ and Theorem~\ref{T:dct}. The last one uses the fact that $T_n$ is affine and continuous.

(R) The construction of the sequence $(T_n)$ is analogous, we indicate the differences.
First, $\{e_n\setsep n\in\en\}$ will be a dense subset of $E$. Further, $\M^1(X)$ will be everywhere replaced by the set
$Y=\{\mu\in \M(X)\setsep\mu\mbox{ is odd and }\norm{\mu}\le 1\}$. Then $Y$ is a compact convex symmetric set.
The mappings $\Gamma_n$ and $\ov{\Gamma}_n$ will be defined in the same way. They are lower semicontinuous for the same reason and, moreover, their graphs are convex and symmetric.
Using Theorem~\ref{T:selekcelsc} we obtain a continuous odd affine selection $T_n$ from $\ov{\Gamma}_n$.

We continue by showing that $T_n\to T$ on $\ext X$. Fix $x\in \ext X$. Let $\mu$ be any cluster point of the sequence $(T_n(x))$ in $Y$. By the very definition of $\Gamma_n$ and $\ov{\Gamma}_n$, $\mu(e_i)=x(e_i)$ for $i\in\en$. By density of the sequence $(e_i)$ in $E$ we get that $\mu(e)=x(e)$ for all $e\in E$. It follows from Lemma~\ref{l:odd-miry2} that there is a probability measure $\nu$ representing $x$ such that $\odd\nu=\mu$. Since $x$ is an extreme point of $X$, necessarily $\nu=\ep_{x}$ and it is maximal, thus $\mu=T(x)$. It follows that $T_n(x)\to T(x)$.

Finally, we will prove that $T_n\to T$ on $X$. Fix any $x\in X$ and $f\in C(X,\er)$.
Let $\sigma$ be a maximal probability representing $x$. Then
$$T(x)(f)=\int_X T(y)(f)\di\sigma(y)=\lim_{n\to\infty} \int_X T_n(y)(f)\di\sigma(y)=\lim_{n\to\infty} T_n(x)(f).$$
The first equality follows by the strong affinity of $T$ (see Lemma~\ref{l:te-sa}). To verify the second one we use that $\sigma$ is maximal, hence supported by $\ext X$,  the already proved fact that $T_n\to T$ on $\ext X$ and Theorem~\ref{T:dct}. The last one uses the fact that $T_n$ is affine and continuous.

(C) The proof in the complex case is completely analogous to the real case. Instead of the set of odd measures we consider the set of anti-homogeneous measures. If we define $\Gamma_n$ and $\ov{\Gamma}_n$ by the same formula, their graphs are clearly absolutely convex and we can find a homogeneous affine continuous selection $T_n$ of $\ov{\Gamma}_n$.
The proof that $T_n\to T$ is analogous to the real case, we just use Lemma~\ref{l:hom-miry2}.
\end{proof}

We remark that the metrizability assumption was used in the previous proof in an essential way.
First, we used the existence of a countable dense set in $\fra(X)$ (or in $E$) and secondly, we used the selection theorem Theorem~\ref{T:selekcelsc} which works for mappings with values in a \fr space,
hence we need the metrizability of the respective set of measures. However, as we know by Lemma~\ref{l:bauer}, if $\ext X$ is moreover closed, $T$ is even continuous and no metrizability assumption is needed. So, it is natural to ask how far one can go in this direction. It follows from Lemma~\ref{L:fra-0}(iii) below that if $T$ is of class $\fra_1$, necessarily $Tf$ is Baire-one for any scalar continuous function $f$ on $X$. This is the case if $\ext X$ is Lindel\"of (by \cite{Jel,lusp23,lusp-complex})
but not only in this case  (see  \cite[Theorem 4]{kalenda-bpms}). The following example shows, in particular, that the Lindel\"of property of $\ext X$ is not a sufficient condition for $T$ being of class $\fra_1$.

\begin{example}\label{ex:dikobraz}
There	are simplices $X_1$ and $X_2$ with the following properties.
\begin{itemize}
	\item[(a)] $\ext X_1$ is Lindel\"of and $\ext X_2$ is an uncountable discrete set.
	\item[(b)] The function $x\mapsto \delta_x(f)$ is Baire-one for continuous $f:X_i\to\er$ ($i=1,2$).
	\item[(c)] The mapping $T:x\mapsto \delta_x$ is not in $\bigcup_{\alpha<\omega_1}\C_\alpha(X_i,\M^1(X_i))$ ($i=1,2$).
\end{itemize}
\end{example}

\begin{proof} We will use the well-known construction of `porcupine simplices' which was used for example in \cite{kalenda-bpms}. Let $A\subset [0,1]$ be an uncountable set. Let
$$K=([0,1]\times\{0\})\cup (A\times\{-1,1\})$$
equipped with the following topology. The points from $A\times\{-1,1\}$ are isolated and a basis of neighborhoods of a point $(x,0)\in [0,1]\times\{0\}$ is formed by the sets of the form
$$(U\times\{-1,0,1\})\cap K \setminus\{(x,1),(x,-1)\},$$
where $U$ is a standard neighborhood of $x$ in $[0,1]$. Then $K$ is a compact space. Let
$$\A=\{f\in \C(K,\er)\setsep f(x,0)=\frac12(f(x,-1)+f(x,1))\mbox{ for each }x\in A\}$$
and let $X=\{\xi\in\A^*\setsep\|\xi\|=1\ \&\ \xi(1)=1\}$ be equipped with the weak$^*$ topology.
Then $X$ is a simplex. Moreover, $K$ canonically homeomorphically embeds into $X$ (as evaluation mappings), in this way $\ext X$ is identified with $((K\setminus A)\times\{0\})\cup (A\times\{-1,1\})$ .

The assertion (b) is valid for any such $X$ by \cite[Theorem 1]{kalenda-bpms}. If we take $A=[0,1]$, then $\ext X$ is uncountable discrete (see \cite[Theorem 4]{kalenda-bpms}); and there is an uncountable $A$ such that $\ext X$ is Lindel\"of
(by \cite[Theorem 2]{kalenda-bpms} it is enough if $A$ contains no uncountable compact subset, cf. \cite[p. 69]{kalenda-bpms}).

Finally, we will prove that (c) is valid for any $X$ of the described form. We consider $K$ canonically embedded in $X$. For any $a\in A$ set $f_a=\chi_{\{(a,1)\}}-\chi_{\{(a,-1)\}}$. Then $f_a\in\A$. Let us define two sets in $X$ by the formula:
$$\begin{aligned}
U&=\{\xi\in X\setsep \exists a\in A: \xi(f_a)>\frac12\},\\
H&=\{\xi\in X\setsep \forall a\in A: \xi(f_a)\ge-\frac12\}.
\end{aligned}$$
Then $U$ is open and $H$ is closed. Moreover, clearly $A\times\{1\}\subset U$ and $H\cap( A\times\{-1\})=\emptyset$.
Further, $U\subset H$. Indeed, pick any $\xi\in U$ and fix $a\in A$ such that $\xi(f_a)>\frac12$. Fix any $b\in A$. We will show that $\xi(f_b)\ge-\frac12$. If $b=a$ the inequality is obvious. So, suppose $b\ne a$. Since the function $f_a-f_b$ belongs to $\A$ and has norm one, we have
$$1\ge\xi(f_a-f_b)=\xi(f_a)-\xi(f_b)>\frac12-\xi(f_b)$$
and the inequality follows.

Now consider the following system of mappings
$$\F=\{S\colon X\to\M^1(X)\setsep \{a\in A\setsep S(a,-1)(H)<S(a,0)(U)\}\mbox{ is countable}\}.$$

First observe that $\F$ contains all continuous mappings. Indeed, suppose that $S$ is continuous. Fix any $q\in\qe$ and set
$$M_q= \{a\in A\setsep S(a,-1)(H)<q<S(a,0)(U)\}.$$
Given $a\in M_q$, we have $S(a,0)(U)>q$. Since $U$ is open, the mapping $\mu\mapsto\mu(U)$ is lower semicontinuous on $\M^1(X)$, hence the set $\{\mu\in\M^1(X)\setsep\mu(U)>q\}$ is open. Therefore there is a neighborhood $W$ of $a$ in $[0,1]$ such that for any $x\in (W\times\{-1,0,1\})\cap K\setminus\{(a,-1),(a,1)\}$ we have $S(x)(U)>q$. In particular, $W\cap M_q=\{a\}$ (recall that $H\supset U$ and so $S(b,-1)(H)>q$ for $b\in W\setminus\{a\}$). It follows that each point of $M_q$ is isolated, hence $M_q$ is countable. Therefore $\bigcup_{q\in\qe}M_q$ is countable as well, hence $S\in\F$.

Further observe that $\F$ is closed with respect to pointwise limits of sequences. Let $(S_n)$ be a sequence in $\F$ pointwise converging to a mapping $S$. By the definition of the system $\F$ there is a countable set $C\subset A$ such that  $S_n(a,-1)(H)\ge S_n(a,0)(U)$ for each $n\in\en$ and each $a\in A\setminus C$. We will show that $S(a,-1)(H)\ge S(a,0)(U)$ for $a\in A\setminus C$ as well. So, fix $a\in A\setminus C$ and suppose that $S(a,-1)(H) < S(a,0)(U)$. Fix a number $q$ such that
$S(a,-1)(H) < q < S(a,0)(U)$. Since $S_n(a,0)\to S(a,0)$ and $U$ is open, we can find $n_0\in\en$ such that for each $n\ge n_0$ we have $S_n(a,0)(U)>q$. Since $a\in A\setminus C$, we get (for $n\ge n_0$) $S_n(a,-1)(H)\ge S_n(a,0)>q$. Since $H$ is closed and $S_n(a,-1)\to S(a,-1)$, we get $S(a,-1)(H)\ge q$, a contradiction.

We conclude that $\bigcup_{\alpha<\omega_1}\C_\alpha(X,\M^1(X))\subset\F$. Finally, the mapping $T$ does not belong to $\F$, since for any $a\in A$ we have
$$\begin{aligned}
T(a,0)(U)&=\delta_{(a,0)}(U)=\frac12(\ep_{(a,1)}(U)+\ep_{(a,-1)}(U))=\frac12,\\
T(a,-1)(H)&=\delta_{(a,-1)}(H)=\ep_{(a,-1)}(H)=0.
\end{aligned}$$
\end{proof}

%%%%%%%%%%%%%%%%%%%%%%%%%%%
%%%%%%%%% KONEC SEKCE %%%%%%%

\section{Strongly affine Baire mappings}\label{Sec:affbaire}

The aim of this section is to prove Theorem~\ref{T:aff-baire}. The proof will be done in two steps. Firstly, we give the proof in case $X$ is metrizable
with the use of Theorem~\ref{T:dilation}. Secondly, we reduce the general case to the metrizable case.

In the proof of the metrizable case we will need the following notation.
Let $X$ be a compact convex set, $F$ a \fr space over $\ef$, $U:X\to B_{\M(X,\ef)}$ be a mapping and $f:X\to F$ be
a bounded Baire mapping. Then we define a mapping $Uf:X\to F$ by
$$Uf(x)=\int_X f\di U(x),\quad x\in X.$$
The mapping $Uf$ is well defined due to Lemma~\ref{L:integr-for-baire}.

\begin{lemma}\label{L:fra-0} Using the above notation, the following assertions hold.
\begin{itemize}
	\item[(i)] If $U$ and $f$ are continuous, then $Uf$ is continuous as well.
	\item[(ii)] If $U$ is strongly affine, then $Uf$ is strongly affine as well.
	\item[(iii)] If $U\in\fra_\alpha(X,B_{\M(X,\ef)})$ and $f\in\C_\beta(X,F)$ is a bounded mapping, then $Uf\in \fra_{\alpha+\beta}(X,\ov{\aco} f(X))$.
\end{itemize}
 \end{lemma}

\begin{proof}
(i) Since $f$ is continuous, $f(X)$ is a compact subset of $F$. Thus $L=\ov{\aco} f(X)$ is an absolutely convex compact subset of $F$ (see, e.g., \cite[Proposition 6.7.2]{jarchow}). By Lemma~\ref{L:norma}(c), $Uf(X)\subset L$. We need to show that $Uf$ is continuous. To this end, let
$\tau\in F^*$ be given.
Then
\[
\tau(Uf(x))=\tau\left(\int_{X} f\di U(x)\right)=\int_{X} \tau\circ f\di U(x)=U(x)(\tau\circ f), \quad x\in X.
\]
Since $U$ is continuous and $\tau\circ f\in\C(X,\ef)$, the mapping $x\mapsto U(x)(\tau\circ f)$ is a continuous function on $X$. Thus $Uf\colon X\to (L,\text{weak})$ is continuous. Since $L$ is compact, the original topology of $F$ coincides on $L$ with the weak topology. Hence $Uf\in\C(X,L)$.

(ii) Let us first suppose that $f$ is continuous.
Let $\mu$ be any Radon probability on $X$. Then for any $\tau\in F^*$ we have
$$
\begin{aligned}
\int_X \tau(Uf(x))\di\mu(x)&=
\int_X \tau(\int_X f \di U(x))\di\mu(x)
=\int_X \int_X \tau\circ f \di U(x) \di\mu(x)
\\&=\int_X U(x)(\tau\circ f)\di\mu(x)
= U(r(\mu))(\tau\circ f)
=\int_X \tau\circ f\di U(r(\mu))
\\&=\tau(Uf(r(\mu))).\end{aligned}$$
Thus $\tau\circ Uf$ is strongly affine for each $\tau\in F^*$, so $Uf$ is strongly affine by Fact~\ref{fact1}.

The general case follows by transfinite induction on the class using Corollary~\ref{C:baire} and the following observation: If $(f_n)$ is a bounded sequence of Baire mappings pointwise converging to a mapping $f$, then for each $x\in X$ and each $\tau\in F^*$
we have
$$\tau(Uf_n(x))=\int_X \tau \circ f_n\di U(x)\to\int_X \tau \circ f\di U(x)=\tau(Uf(x)).$$

(iii) Set $L=\ov{\aco} f(X)$. By Lemma~\ref{L:norma}(c) we have $Uf(X)\subset L$.

Let us give the proof first for $\beta=0$, i.e., in case $f$ is continuous.
The case $\alpha=0$ follows from (i) and (ii). We continue by transfinite induction.
To proceed it is enough to observe that $U_n\to U$ pointwise on $X$ implies $U_nf\to Uf$ pointwise on $X$. So, suppose that $U_n\to U$ pointwise on $X$.
Fix any $x\in X$. Then $U_n(x)\to U(x)$. For any $\tau\in F^*$ we have
$$\tau(U_n f(x))=U_n(x)(\tau\circ f)\to U(x)(\tau\circ f)=\tau(Uf(x)).$$
Thus $U_nf(x)\to Uf(x)$ weakly in $F$. But since the sequence is contained in the compact set $L$, we deduce that $U_nf(x)\to Uf(x)$ in $F$. This completes the proof for $\beta=0$.

Suppose that $\gamma>0$ is given such that the assertion is valid for any $\beta<\gamma$. Suppose that $f\in\C_\gamma(X,F)$ is bounded. Then $L=\ov{\aco}(F)$ is bounded and $f\in\C_\gamma(X,L)$ due to Corollary~\ref{C:baire}. So, fix a sequence $(f_n)$ in $\bigcup_{\beta<\gamma} \C_\beta(X,L)$ pointwise converging to $f$. Then $Uf_n\in\bigcup_{\beta<\gamma} \fra_{\alpha+\beta}(X,L)$ by the induction hypothesis. Further, $Uf_n\to Uf$ pointwise by Theorem~\ref{T:dct}, hence $Uf\in\fra_{\alpha+\gamma}(X,L)$.
This completes the proof.
\end{proof}

Now we are ready to complete the first step:

\begin{proof}[Proof of Theorem~\ref{T:aff-baire} in case $X$ is metrizable.]

(S) Let $f$ be strongly affine and $f\in\C_\alpha(X,F)$. Then $f(x)=\delta_x(f)$ for each $x\in X$. Since the mapping $x\mapsto\delta_x$ belongs to $\fra_1(X,\M_1(X))$ by Theorem~\ref{T:dilation},  by Lemma~\ref{L:fra-0}(iii) we conclude that $f\in\fra_{1+\alpha}(X,F)$.

(R) If $f$ is odd and strongly affine, then $f=Tf$. Since $T\in\fra_1(X,B_{\M_{\odd}(X,\er)})$ by  Theorem~\ref{T:dilation}, we conclude by Lemma~\ref{L:fra-0}(iii) and Lemma~\ref{l:odd}(c). If $f$ is not odd, then $f=f(0)+(f-f(0))$. Since $f-f(0)$ is odd, we get $f-f(0)\in\fra_{\odd,1+\alpha}(X,F)$, thus $f\in\fra_{1+\alpha}(X,F)$.

(C) If $F$ is complex and $f$ is homogeneous and strongly affine, then $f=Tf$. Since $T\in\fra_1(X,B_{\M_{\ahom}(X,\ce)})$ by  Theorem~\ref{T:dilation}, we conclude by Lemma~\ref{L:fra-0}(iii) and Lemma~\ref{l:homaff}(b). If $F$ is complex and $f$ is not homogeneous, then we can write
$$f(x)=f(0)+u(x)+\ov{v(x)},\quad x\in X,$$
where $u$ and $v$ are homogeneous and strongly affine and, moreover, $u=\hom f$ (see Lemma~\ref{l:homaff}(a)). Then $u,v\in\C_\alpha(X,F)$ by Lemma~\ref{l:hom}(c),
hence $u,v\in \fra_{\hom,1+\alpha}(X,F)$. It follows that $f\in\fra_{1+\alpha}(X,F)$. Finally, if $F$ is a real \fr space, denote by $F_{\ce}$ its complexification. Then $F$ is a real-linear subspace of $F_{\ce}$, thus $f\in\C_\alpha(X,F_{\ce})$, so $f\in\fra_{1+\alpha}(X,F_{\ce})$. Since the canonical projection of $F_{\ce}$ onto $F$ is continuous and real-linear (hence affine), it is clear that $f\in\fra_{1+\alpha}(X,F)$.

\smallskip

Finally, consider the case $\alpha=1$. In all the cases we get as above by the use of Lemma~\ref{L:fra-0}(iii) that $f\in\fra_2(X,\ov{\aco f(X)})$. Hence, Corollary~\ref{C:c21} yields $f\in\fra_1(X,\ov{\aco f(X)})$.
\end{proof}

To prove the general statement we need a reduction to the metrizable case. The first step is the following lemma.

\begin{lemma}
\label{selekce5}
Let $E$ be a Banach space, $X\subset B_{E^*}$ a weak$^*$ compact convex set and $L$ be a convex subset of a \fr space $F$.
Let $f\colon X\to L$ be a mapping of type $\C_\alpha(X,L)$ for some $\alpha\in [0,\omega_1)$. Then there exist a closed separable subspace
$E_1\subset E$ and $g\colon \pi(X)\to L$ such that $g\in \C_\alpha(\pi(X),L)$ and $f=g\circ \pi$. (Here $\pi\colon E^*\to {E_1^*}$ denotes the restriction mapping.)
\end{lemma}

\begin{proof}
Assume first that $\alpha=0$, i.e., that $f$ is continuous. Then $f(X)$ is a compact subset of $F$, hence it is a compact metrizable space. It follows that there is a homeomorphic injection $p\colon f(X)\to\er^\en$. Further, since $X$ is equipped with the weak$^*$-topology, it can be canonically embedded into the cartesian product $\ef^E$. The continuous mapping $p\circ f\colon X\to\er^\en$ can be extended to a continuous mapping $g\colon\ef^E\to\er^\en$ (by \cite[Theorem 3.1.7]{engelking}). Further, by \cite[Theorem 4]{Ross-Stone} (see also \cite[Problem 2.7.12(d)]{engelking}) there is a countable set $C\subset E$ such that
$$u,v\in \ef^E, u\r_{C}=v\r_{C} \Rightarrow g(u)=g(v).$$
Hence we can take $E_1$ to be the closed linear span of $C$. Then $E_1$ is a closed separable subspace of $E$ and, moreover,
$\pi(x_1^*)=\pi(x_2^*)$ for some $x_1^*,x_2^*\in X$ implies $(p\circ f)(x_1^*)=(p\circ f)(x_2^*)$, and so $f(x_1^*)=f(x_2^*)$. Thus there is a mapping $g:\pi(X)\to L$ with $f=g\circ\pi$. Since $f$ is continuous and $\pi$ is a closed continuous mapping, $g$ is continuous.

Assume now that $\alpha\in (0,\omega_1)$ and $f\colon X\to L$ of type $\C_\alpha(X,L)$ is given. We select a countable family $\F=\{f_n\setsep n\in\en\}$ in $\C(X,L)$ such that $f\in \F_\alpha$.
For each $n\in\en$ we find using the previous step a countable set $C_n\subset E$ such that for any $x_1^*,x_2^*\in X$ we have
$$x_1^*\r_{C_n}=x_2^*\r_{C_n}\Rightarrow f_n(x_1^*)=f_n(x_2^*).$$
Let $E_1$ be the closed linear span of $\bigcup_n C_n$. Then there are mappings $g:\pi(X)\to L$ and $g_n:\pi(X)\to L$ for $n\in\en$
such that $f=g\circ\pi$ and $f_n=g_n\circ\pi$. Similarly as above $g_n$ are continuous and, moreover, it is easy to check by transfinite induction
that $g\in\left(\{g_n:n\in\en\}\right)_\alpha$, thus $g\in\C_\alpha(\pi(X),L)$.
\end{proof}

The next step is the following lemma on cofinality.

\begin{lemma}
\label{L:rich} \
\begin{itemize}
\item[(S)] Let $X$ be a simplex and $f:X\to X_1$ an affine continuous surjection of $X$ onto a metrizable compact convex set. Then there is a metrizable simplex $X_2$ and affine continuous surjections $f_1:X_2\to X_1$ and $f_2:X\to X_2$ such that $f=f_1\circ f_2$.
	\item[(R,C)] Let $E$ be an $L_1$-predual over $\ef$ and $E_1$ be its separable subspace. Then there exists a separable $L_1$-predual $E_2$ satisfying $E_1\subset E_2\subset E$.
\end{itemize}
\end{lemma}

\begin{proof} (S) In this proof we will denote by $\fra(X)$ the space $\fra(X,\er)$ and similarly for other compact convex sets.
Let $f^*\colon\fra(X_1)\to\fra(X)$ denote the canonical isometric embedding defined by $f^*(u)=u\circ f$ for $u\in\fra(X_1)$. Let $E$ be a (for a while arbitrary) closed subspace of $\fra(X)$ containing $f^*(\fra(X_1))$. Denote by  $\pi_2:\fra(X)^*\to E^*$ the canonical restriction map. Define $\pi_1:E^*\to \fra(X_1)^*$ by $\pi_1(x^*)(u)=x^*(f^*(u))$, $x^*\in E^*$, $u\in\fra(X_1)$. 
Further, let $\kappa_X:X\to\fra(X)^*$ be the canonical evaluation mapping, i.e., $\kappa_X(x)(u)=u(x)$ for $x\in X$ and $u\in\fra(X)$. Similarly, $\kappa_{X_1}$ denotes the analogous mapping for $X_1$. Set $X_2=\pi_2(\kappa_{X}(X))$, $f_2=\pi_2\circ\kappa_X$ and $f_1=(\kappa_{X_1})^{-1}\circ \pi_1\r_{X_2}$. Then $X_2$ is a compact convex set, and $f_1:X_2\to X_1$ and $f_2:X\to X_2$ are affine continuous surjections satisfying $f=f_1\circ f_2$. So, to complete the proof it is enough to choose $E$ in such a way that $X_2$ is a metrizable simplex.

Observe that $E$ is canonically isometric to $\fra(X_2)$. More precisely, if we consider the isometric embedding $f_2^*\colon \fra(X_2)\to\fra(X)$
defined by $f_2^*(u)=u\circ f_2$ for $u\in\fra(X_2)$, then $E=f_2^*(\fra(X_2))$. Indeed, if $v\in E$, we define $u\in\fra(X_2)$ by $u(x^*)=x^*(v)$ for $x^*\in X_2$. Then for each $x\in X$ we have
$$f_2^*(u)(x)=u(f_2(x))=f_2(x)(v)=\pi_2(\kappa_X(x))(v)=\kappa_X(x)(v)=v(x),$$
thus $v\in f_2^*(\fra(X_2))$. Conversely, let $u\in \fra(X_2)$. Then $f_2^*(u)=u\circ f_2\in\fra(X)$. If $u\circ f_2\notin E$, by the Hahn-Banach separation theorem there is $x^*\in \fra(X)^*$ such that $x^*\r_{E}=0$ and $x^*(u\circ f_2)\ne 0$. By \cite[Proposition 4.31(a,b)]{lmns} we have $x^*=c_1\kappa_X(x_1)-c_2\kappa_X(x_2)$ for some $x_1,x_2\in X$ and $c_1,c_2$ nonnegative real numbers. Since the space $E$ contains the constant functions, we get $c_1=c_2$. Hence without loss of generality $c_1=c_2=1$. Then we get $v(x_1)=v(x_2)$ for each $v\in E$. It follows that $\pi_2\circ\kappa_X(x_1)=\pi_2\circ\kappa_X(x_2)$, thus $f_2(x_1)=f_2(x_2)$. So,
$$x^*(u\circ f_2)=(u\circ f_2)(x_1)-(u\circ f_2)(x_2)=0,$$
a contradiction.

Hence, $X_2$ is metrizable provided $E$ is separable, and $X_2$ is a simplex provided $E$ satisfies the weak Riesz interpolation property
(see \cite[Corollary II.3.11]{alfsen}). Recall that $E$ has the weak Riesz interpolation property if, whenever $u_1,u_2,v_1,v_2 \in E$ are such that
$\max\{u_1,u_2\}<\min\{v_1,v_2\}$, then there is $w\in E$ with $\max\{u_1,u_2\}<w<\min\{v_1,v_2\}$. The inequalities are considered in the pointwise meaning. Further, it is clear that it is enough to check this property for $u_1,u_2,v_1,v_2$ belonging to a norm-dense subset of $E$. Thus we can construct $E$ by a standard inductive procedure: We construct countable sets $A_0\subset B_1\subset A_1\subset B_2\subset\cdots \fra(X)$ as follows:
\begin{itemize}
	\item $A_0$ is a countable dense subset of $f^*(\fra(X_1))$.
	\item Whenever $u_1,u_2,v_1,v_2 \in A_{n-1}$ are such that
$\max\{u_1,u_2\}<\min\{v_1,v_2\}$, then there is $w\in B_n$ with $\max\{u_1,u_2\}<w<\min\{v_1,v_2\}$.
\item $A_n$ is the $\qe$-linear span of $B_n$.
\end{itemize}
Then we can set $E=\ov{\bigcup_n A_n}$.

(R,C) This is proved in \cite[\S23, Lemma 1]{lacey}.
\end{proof}

The final ingredient is the following lemma.

\begin{lemma}
\label{perfectaff}
Let $\pi\colon K\to L$ be an affine continuous surjection of a compact convex set $K$ onto a compact convex set $L$. Let $g\colon L\to F$ be a universally measurable mapping from $L$ to a \fr space $F$. Then $g$ is strongly affine if
and only if $g\circ \pi$ is strongly affine.
\end{lemma}

\begin{proof} In case $F=\ef$ it is proved in \cite[Proposition 5.29]{lmns}. The vector-valued case then follows immediately from Fact~\ref{fact1}.
\end{proof}

Now we can  complete  the proof:

\begin{proof}[Proof of Theorem~\ref{T:aff-baire} in the general case.]
(S) Let $f\in\C_\alpha(X,F)$ be strongly affine. Set $E=\fra(X,\er)$ and $\kappa_X:X\to E^*$ be the canonical embedding.
If we apply Lemma~\ref{selekce5} to $\kappa_X(X)$ in place of $X$ and $f\circ\kappa_X^{-1}$ in place of $f$, we get a separable
space $E_1\subset E$ and $g\in\C_\alpha(\pi(\kappa_X(X)),F)$ with $f\circ\kappa_X^{-1}=g\circ \pi$ (where $\pi:E^*\to E_1^*$ is the restriction mapping). Set $X_1=\pi(\kappa_X(X))$ and $h=\pi\circ\kappa_X$. Then $X_1$ is a metrizable compact convex set and $h$ is an affine continuous surjection of $X$ onto $X_1$. By Lemma~\ref{L:rich} there is a metrizable simplex $X_2$ and affine continuous surjections $h_1:X_2\to X_1$ and $h_2:X\to X_2$ with $h=h_1\circ h_2$. Then $g\circ h_1\in\C_\alpha(X_2,F)$ and $g\circ h_1$ is strongly affine by Lemma~\ref{perfectaff}. Thus by the metrizable case, $g\circ h_1\in \fra_{1+\alpha}(X_2,F)$ It follows that $f=g\circ h_1\circ h_2\in\fra_{1+\alpha}(X,F)$.

(R), (C) Given $f$ as in the premise, we use Lemma~\ref{selekce5} to find a separable subspace $E_1\subset E$ and $g\colon B_{E_1^*}\to F$ in $\C_\alpha(B_{E_1^*},F)$ satisfying $f=g\circ \pi$ ($\pi\colon {E^*}\to {E_1^*}$ is again the restriction mapping).
By Lemma~\ref{L:rich} we can find a separable $L_1$-predual $E_2$ satisfying $E_1\subset  E_2\subset E$. Denote $\pi_1:E_2^*\to E_1^*$ and $\pi_2:E^*\to E_2^*$ the restriction maps. Then $g\circ\pi_1\in\C_{\alpha}(B_{E_2^*},F)$. Moreover,  $g\circ\pi_1$ is strongly affine by Lemma~\ref{perfectaff}. Hence by the metrizable case we get $g\circ\pi_1\in\fra_{1+\alpha}(B_{E_2^*},F)$ (or $g\circ\pi_1\in\fra_{\odd,1+\alpha}(B_{E_2^*},F)$ or $g\circ\pi_1\in\fra_{\hom,1+\alpha}(B_{E_2^*},F)$
in the special cases). Since $f=g\circ\pi_1\circ\pi_2$, the proof is complete.

\smallskip

If $\alpha=1$ we obtain by the same procedure that $f\in\fra_1(X,F)$.
\end{proof}

Finally, let us settle the remaining special case:

\begin{proof}[Proof of Theorem~\ref{T:aff-baire} in case $\ext X$ is an $F_\sigma$-set.]
If $\ext X$ is $F_\sigma$, the result follows from Theorem~\ref{T:dirichlet} proved below. More precisely:

(S) Let $f\in\C_\alpha(X,F)$ be strongly affine. Then $f\r_{\ext X}\in\C_\alpha(\ext X,F)$, hence by Theorem~\ref{T:dirichlet} this function can be extended to a function $g\in\fra_\alpha(X,F)$. Since $f=g$ on $\ext X$, both functions are strongly affine and each maximal measure is supported by $\ext X$ we conclude that $f=g$ on $X$, i.e., $f\in\fra_\alpha(X,F)$.

(R) Let $f\in\C_\alpha(X,F)$ be strongly affine. Then $f=f(0)+(f-f(0))$ and the function $f-f(0)$ is odd, strongly affine and belongs to $\C_\alpha(X,F)$. Using Theorem~\ref{T:dirichlet} as in the case (S) we get $f-f(0)\in\fra_\alpha(X,F)$, thus $f\in\fra_\alpha(X,F)$.

(C) Let $f\in\C_\alpha(X,F)$ be strongly affine. Then $f=f(0)+u+\ov{v}$, where $u=\hom f$ and $v$ are homogeneous, affine and belong to $\C_\alpha(X,F)$ (see Lemma~\ref{l:homaff}(a)).
Using Theorem~\ref{T:dirichlet} as in the case (S) we get $u,v\in\fra_\alpha(X,F)$, thus $f\in\fra_\alpha(X,F)$.
\end{proof}

\section{Extensions of Baire mappings}\label{Sec:dirichlet}

In this section we will prove Theorem~\ref{T:dirichlet}. We will proceed in several steps, imitating, generalizing and simplifying the approach of \cite{lusp23}. The strategy is the following:
\begin{itemize}
	\item Given a bounded Baire function $f\colon\ext X\to F$, we extend it to a bounded Baire function $h:X\to F$.
	(We do not require affinity and we do not control the class of $h$.)
	\item We take the function $Th(x)=\int_X h\dd T(x)$, $x\in X$, used in the previous section, and we prove that it is the unique strongly affine extension of $f$.
	\item We show that $Th$ is of the right affine class.
\end{itemize}

The first step is made by the following lemma which is a vector-valued variant of \cite[Lemma 2.8]{lusp23}
with a simplified proof.

\begin{lemma}
\label{L:extension}
Let $X$ be a compact convex set with $\ext X$ \lin. Let $F$ be a \fr space over $\ef$,  and $f$ be a bounded function in $\C_\alpha(\ext X,F)$ for some $\alpha<\omega_1$. Let $L=\ov{\co}f(\ext X)$. Then there exists a  Baire
measurable function $h\colon X\to L$ extending $f$.

If $f\in\C_1(\ext X,F)$, $h$ may be chosen from $\C_1(X,L)$.
\end{lemma}

\begin{proof} We will prove the result by transfinite induction on $\alpha$. Suppose first that $\alpha=1$, i.e., that $f\in\C_1(\ext X,F)$. Since $L$ is separable and completely metrizable, by \cite[Theorem 30 and Proposition 28]{kalenda-spurny}
there is an extension $h:X\to L$ which is $\Sigma^b_2(X)$-measurable. Lemma~\ref{L:baire-n} now implies that $h\in\C_1(X,L)$.

Assume now that $\alpha>1$ and the assertion is valid for all $\beta<\alpha$. Suppose that $f\in\C_\alpha(\ext X,F)$ is a bounded mapping and let $L$ be as above. Then $f\in \C_\alpha(\ext X,L)$ by
Corollary~\ref{C:baire} (note that $\ext X$ is normal, being Lindel\"of and regular), and thus there exist mappings $f_n\in\bigcup_{\beta<\alpha}\C_\beta(\ext X,L)$ converging pointwise to $f$ on $\ext X$. Let $h_n\colon X\to L$ be their Baire measurable extensions and let
\[
C=\{x\in X\setsep (h_n(x))\text{ converges}\}.
\]
Let $\rho$ be a compatible complete metric on $F$. Then
\[
C=\{x\in X\setsep \forall k\in\en\ \exists l\in\en\ \forall m_1,m_2\ge l\colon \rho(h_{m_1}(x),h_{m_2}(x))<\frac1k\},
\]
which gives that $C$ is a Baire subset of $X$. Let $z$ be an arbitrary element of $L$. Then the function
\[
h(x)=\begin{cases}
\lim_{n\to \infty} h_n(x),& x\in C,\\
z,& x\in X\setminus C,
\end{cases}
\]
is the required extension.
\end{proof}

The next lemma enables us to deduce the vector-valued version from the scalar one.

\begin{lemma}
\label{L:meas}
Let $K$ be a compact space and $L$ be a separable convex subset of a \fr space $F$.
Let $f\colon K\to L$ satisfy $\tau\circ f\in \C_{\alpha}(K,L)$ for each $\tau\in F^*$. Then $f\in \C_{\alpha+1}(K,L)$.
\end{lemma}

\begin{proof} Without loss of generality we can suppose that $F$ is separable. Let $f$ satisfy the assumptions. By Lemma~\ref{L:baire-n}, the composition $\tau\circ f$ is $\Sigma_{\alpha+1}^b(K)$-measurable for each $\tau\in F^*$. Since $(\tau\circ f)^{-1}(U)=f^{-1}(\tau^{-1}(U))$ for any $U\subset \ef$ and $F$, being separable, is hereditarily Lindel\"of in the weak topology, we get that $f$ is  $\Sigma_{\alpha+1}^b$-measurable as a mapping from $K$ to the weak topology of $F$. It follows from Lemma~\ref{L:meas-0} that it is  $\Sigma_{\alpha+2}^b(K)$-measurable as a mapping from $K$ to the original topology of $F$.
Thus $f\in\C_{\alpha+1}(K,L)$ by Lemma~\ref{L:baire-n}.
\end{proof}

\begin{lemma}\label{L:diri} Let $X$ be a compact convex set with $\ext X$ Lindel\"of, $F$ a \fr space and $h\colon X\to F$ a bounded Baire mapping.
Suppose moreover that one of the following conditions is satisfied.
\begin{itemize}
	\item[(S)] $X$ is simplex.
	\item[(R)] $X=(B_{E^*},w^*)$, where $E$ is a real $L_1$-predual and $h\r_{\ext X}$ is odd.
	\item[(C)] $X=(B_{E^*},w^*)$, where $E$ is a complex $L_1$-predual and $h\r_{\ext X}$ is homogeneous.
\end{itemize}
Then $Th$ is the unique strongly affine mapping which coincides with $h$ on $\ext X$.
Moreover, $Th$ is a Baire mapping and it is odd in case (R) and homogeneous in case (C).
\end{lemma}

\begin{proof} Let us show first the uniqueness. Suppose that $g_1,g_2$ are two strongly affine mappings which coincide with $h$ on $\ext X$. Fix any $\tau\in F^*$. Then, for each $i=1,2$ the function $\tau\circ g_i$ is strongly affine and $\tau\circ g_i\r_{\ext X}$ is a Baire mapping (as it coincides with $h\r_{\ext X}$), hence $\tau\circ g_i$ is a Baire mapping by \cite[Theorem 5.2]{lusp}. In particular, 
the set $\{x\in X\setsep \tau(g_1(x))=\tau(g_2(x))\}$ is a Baire set containing $\ext X$. Given any $x\in X$ let $\mu$ be a maximal probability representing $x$. Then $\mu$ is carried by any Baire set containing $\ext X$, hence
$\tau(g_1(x))=\mu(\tau\circ g_1)=\mu(\tau\circ g_2)=\tau(g_2(x))$. Hence $\tau\circ g_1=\tau\circ g_2$. Since $\tau\in F^*$ is arbitrary, we conclude $g_1=g_2$.

Further, observe that any strongly affine mapping coinciding with $h$ on $\ext X$ is a Baire mapping. Indeed, let $g$ be such a mapping and suppose that $h\in\C_\alpha(X,F)$. Using \cite[Theorem 5.2]{lusp} as in the previous paragraph we get $\tau\circ g\in\C_{1+\alpha}(X,\ef)$ for each $\tau\in F^*$, hence $g\in\C_{1+\alpha+1}(X,F)$ by Lemma~\ref{L:meas}. In particular, $g$ is a Baire mapping.

Finally, $Th$ is strongly affine by Lemma~\ref{l:te-sa} and Lemma~\ref{L:fra-0}(ii). To finish the proof it is enough to
show that $Th$ coincides with $h$ on $\ext X$ and that it enjoys the appropriate symmetry in cases (R) and (C). The fact that $Th$ is a Baire mapping then follows from the previous paragraph. Hence distinguish the three cases:

(S) For $x\in\ext X$ we have $Th(x)=\delta_x(h)=h(x)$ since $\delta_x$ is the Dirac measure supported at $x$.

(R) To show that $g$ is odd it is enough to observe that $T$ is odd (if $x^*\in X$, then $-T(x^*)$ satisfies the conditions (a)--(d) from Fact~\ref{fact:L1} for $-x^*$). Further, for any $x^*\in \ext X$ we have $T(x^*)=\odd\ep_{x^*}$, hence
$$Th(x^*)=\int_X h\di \odd\ep_{x^*}=\int_X\odd h\di\ep_{x^*}=\odd h(x^*)=h(x^*).$$
 
(C) This case is completely analogous to the case (R).
\end{proof}

The final ingredient is the following easy lemma:

\begin{lemma}\label{L:limita} Let $X$ be a compact convex set, $F$ a \fr space and $f$, $f_n$, $n\in\en$, strongly affine Baire mappings defined on $X$ with values in $F$. If the sequence $(f_n)$ is uniformly bounded and converges to $f$ pointwise on $\ext X$, it converges to $f$ pointwise on $X$.
\end{lemma}

\begin{proof} Let $A=\{x\in X\setsep f_n(x)\to f(x)\}$. By the assumption we have $A\supset \ext X$ and, moreover, $A$ is a Baire set (as the functions in question are Baire functions, cf. the proof of Lemma~\ref{L:extension}). Therefore any maximal probability measure is carried by $A$. Fix any $x\in X$ and a maximal probability $\mu$ representing $x$. Using Theorem~\ref{T:dct} we have
$$f(x)=\mu(f)=\lim\mu(f_n)=\lim f_n(x).$$
\end{proof}

Now we are prepared to complete the proof:

\begin{proof}[Proof of Theorem~\ref{T:dirichlet}.] Let us first prove the theorem for $\alpha=1$.
So, suppose that $f\in\C_1(\ext X,F)$ is a bounded mapping. Let $L=\ov{\aco} f(\ext X)$. By Lemma~\ref{L:extension}
there is $h\in\C_1(X,L)$ extending $f$. Set $g=Th$. By Lemma~\ref{L:diri} we know that $g$ is a strongly affine Baire mapping extending $f$. Now let us distinguish the cases:

(S) In the same way as in the proof of Theorem~\ref{T:aff-baire} we find (using Lemmata~\ref{selekce5} and~\ref{L:rich}) 
a metrizable simplex $Y$, an affine continuous surjection $\pi\colon X\to Y$, $\tilde h\in\C_1(Y,L)$ and a Baire mapping $\tilde g\colon Y\to L$ such that $h=\tilde h\circ \pi$ and $g=\tilde g\circ \pi$. By Lemma~\ref{perfectaff} the mapping $\tilde g$ is strongly affine. Moreover, since $\pi(\ext X)\supset\ext Y$, $\tilde g$ coincide with $\tilde h$ on $\ext Y$.
Lemma~\ref{L:diri} then yields  $\tilde g= T\tilde h$. By Theorem~\ref{T:dilation} and Lemma~\ref{L:fra-0}(iii) we get $\tilde g\in\fra_2(Y,L)$. Hence $g\in\fra_2(X,L)$.

The cases (R) and (C) are analogous.

Let us continue with the general case. Lemmata~\ref{L:extension} and~\ref{L:diri} imply that for any (odd, homogeneous) bounded $F$-valued map $f$ on $\ext X$ there is a unique strongly affine map $Sf$ extending $f$. By Lemma~\ref{L:limita} we know that $Sf_n\to Sf$ pointwise whenever $(f_n)$ is a bounded sequence converging to $f$ pointwise on $\ext X$. The already proved case
$\alpha=1$ shows that $Sf\in\fra_2(X,F)$ whenever $f\in\C_1(\ext X,F)$. Hence the result follows by transfinite induction.

\medskip

The case of $\ext X$ being $F_\sigma$ is proved in the next section using Theorem~\ref{T:weakDP}.

\medskip

Finally, suppose that $\ext X$ is closed. For each bounded $f\in\C_\alpha(\ext X,F)$  (odd, homogeneous) we construct the extension $g$ by the same method as above. We will prove that $g\in\fra_\alpha(\ext X,F)$. Since boundary measures are supported by $\ext X$, we have
$$g(x)=\int_{\ext X} f\di T(x),\quad x\in X.$$
Suppose first that $\alpha=0$, i.e., $f$ is continuous. Then we can prove that $g$ is continuous by a minor modification of the proof of Lemma~\ref{L:fra-0}(i) with the help of Lemma~\ref{l:bauer}. The general case then follows by transfinite induction using Lemma~\ref{L:limita}.
\end{proof}

\section{The weak Dirichlet problem for Baire mappings}\label{Sec:wdp}

In this section we will prove Theorem~\ref{T:weakDP}. The proof will be a simplified and generalized version of the proof of the main result of \cite{spurny-wdp}.

\begin{proof}[Proof of Theorem~\ref{T:weakDP}.]
(S) Suppose that $X$, $K$, $F$ and $f$ satisfy the assumptions. Set $H=\ov{\co}(K)$ and $L=\ov{\co}(f(K))$.
By Corollary~\ref{C:baire} we have $f\in\C_\alpha(K,L)$. Fix a countable family $\F=\{f_n\setsep n\in\en\}\subset\C(K,L)$
such that $f\in(\F)_\alpha$.

We define $\varphi\colon K\to F^\en$ by
\[
\varphi(x)=(f_n(x))_{n=1}^\infty,\quad x\in K.
\]
Then $\varphi$ is a continuous mapping of $K$ into $F^\en$.  Further we set
\[
\Gamma(x)=\begin{cases} \{\varphi(x)\}, & x\in K,\\ \ov{\co}\, \varphi(K),& x\in X\setminus K. \end{cases}
\]
Then $\Gamma$ is a  lower semicontinuous mapping with closed values. Moreover, the graph of $\Gamma$ is convex. By Theorem~\ref{T:selekcelsc} there exists a continuous
affine  selection $\gamma\colon X\to F^\en$ of $\Gamma$.

Let $h_n=\pi_n\circ \gamma$, $n\in\en$, where $\pi_n\colon F^\en\to F$ is the $n$-th projection mapping. Then $h_n$ are continuous affine mappings of $X$ into $F$ such that
\[
h_n(x)=\pi_n(\gamma(x))=\pi_n(\varphi(x))=f_n(x),\quad x\in K.
\]
Further, for each $x\in X$ we choose some $r(x)\in H$ such that $h_n(x)=h_n(r(x))$ for each $n\in\en$. If $x\in H$ we set $r(x)=x$. If $x\in X \setminus H$, then $\gamma(x)\in\ov{\co}\,\varphi(K)=\ov{\co}\,\gamma(K)=\gamma(H)$, thus there exists an element $r(x)\in H$ such that $\gamma(x)=\gamma(r(x))$.

Finally, we can construct the required extension by
\begin{equation}
\label{eq:rozsireni}
\hat{f}(x)=\begin{cases} \delta_x(f), & x\in H,\\ \hat{f}(r(x)), & x\in X\setminus H.\end{cases}\end{equation}

First, if $x\in H$, then  $\delta_x$  is supported by $K$, thus $\hat{f}$ is well defined on $H$. Subsequently, $\hat{f}$ is extended to $X$ using the mapping $r$ fixed above. It is clear that $\hat{f}$ is an extension of $f$. Moreover,
\begin{equation}
\label{eq:fhat}
\hat{f}\in\left(\{h_n\setsep n\in\en\}\right)_\alpha\subset\fra_{\alpha}(X,L).\end{equation}
Indeed, the formula \eqref{eq:rozsireni} enables us to assign to each Borel function $g\colon K\to L$ its extension $\hat{g}\colon X\to L$. Note that 
$$\hat{f_n}=h_n \mbox{ for }n\in\en.$$ 
Further, it follows from Theorem~\ref{T:dct} that
$$g_k\to g\mbox{ pointwise on }K\Rightarrow \hat{g_k}\to\hat{g}\mbox{ pointwise on }X.$$
Hence it is easy to prove by transfinite induction on $\beta$ that
$$g\in\left(\{f_n\setsep n\in\en\}\right)_\beta\Rightarrow\hat{g}\in\left(\{h_n\setsep n\in\en\}\right)_\beta,$$
so, in particular, \eqref{eq:fhat} holds.

(R) We assume that $K$ is moreover symmetric and $f$ is odd. We proceed in the same way. The sets $H$ and $L$ are convex and symmetric. The family $\F$ may consist of odd functions by Lemma~\ref{l:odd}. Then $\varphi$ is moreover odd and the graph of $\Gamma$ is convex and symmetric. Thus the selection $\gamma$ may be chosen to be moreover odd and the functions $h_n$ are odd as well. The formula for the extension is similar:
$$\hat{f}(x)=\begin{cases} T(x)(f), & x\in H,\\ \hat{f}(r(x)), & x\in X\setminus H.\end{cases}$$
The rest of the proof is the same, we use the fact that $T(x)$ is supported by $K$ whenever $x\in H$.

(C) The proof is completely analogous to the real case.
\end{proof}

Now we give the proof of the missing part from Theorem~\ref{T:dirichlet}.

\begin{proof}[Proof of Theorem~\ref{T:dirichlet} in case $\ext X$ is $F_\sigma$.]
(S) It is enough to prove that $x\mapsto\delta_x(f)$ is in $\fra_1(X,F)$ for any bounded $f\in \C_1(\ext X,F)$.
Hence choose any bounded $f\in\C_1(\ext X,F)$. Fix a bounded sequence of continuous functions $f_n:\ext X\to F$ pointwise converging to $f$ (this can be done by Corollary~\ref{C:baire}). Let $\ext X=\bigcup_n K_n$, where $(K_n)$ is an increasing sequence of compact sets. By Theorem~\ref{T:weakDP} there are affine continuous maps $h_n$ extending $f_n\r_{K_n}$. Then $h_n$ converge to $f$ pointwise on $\ext X$, thus for each $x\in X$
we have
$$\delta_x(f)=\lim_{n\to\infty}\delta_{x}(h_n)=\lim_{n\to\infty} h_n(x).$$
This completes the proof.

(R) The proof is analogous. Assume that $f$ is moreover odd. We can choose $f_n$ to be odd (Lemma~\ref{l:odd}) and $K_n$ to be symmetric.
In the final computation we use $T(x)$ instead of $\delta_x$.

(C) If $f$ is homogeneous, we can choose $f_n$ to be homogeneous by Lemma~\ref{l:hom}. Moreover, $K_n$ can be chosen to be
homogeneous as well. Indeed, if $K$ is compact, then $\bigcup\{\alpha K\setsep \alpha\in\ce,|\alpha|=1\}$ is compact as a continuous image of the compact set $K\times \{\alpha\in\ce\setsep|\alpha|=1\}$.
\end{proof}

As an another consequence of Theorem~\ref{T:weakDP} we get the following extension theorem.

\begin{thm}
\label{T:weak2}
Let $K$ be a compact subset of a completely regular space $Z$, $F$ be a \fr space and $f\colon K\to F$ be a bounded mapping in $\C_\alpha(K,F)$. Then there exists a mapping $h\colon Z\to F$ in $\C_\alpha(Z,F)$ extending $f$ such
that $h(Z)\subset \overline{\co} f(K)$.
\end{thm}

\begin{proof}
Let $L=\ov{\aco} f(K)$. Let $\beta Z$ be the \v{C}ech-Stone compactification of $Z$ and $X=\M^1(\beta Z)$. Then $X$ is  a simplex and $\ext X$ is canonically identified with $\beta Z$. Hence $K$ is a compact subset of $\ext X$. Therefore the result follows from Theorem~\ref{T:weakDP}.
\end{proof}

\section{Affine version of the Jayne-Rogers selection theorem}

The aim of this section is to prove Theorem~\ref{T:selekceusc}. To this end we need the following lemma on `measure-convexity':

\begin{lemma}
\label{L:sel-convex}
Let $X$ be a compact convex set, $F$ be a \fr space over $\ef$ and $\Gamma\colon X\to F$ be an upper or lower semicontinuous mapping with closed values and convex graph, such that $\Gamma(X)$ is bounded. Let $f\colon X\to F$ be a Baire
measurable selection from $\Gamma$. If $\mu\in\M^1(X)$, then
\[
\mu(f)\in \Gamma(r(\mu)).
\]
\end{lemma}

\begin{proof}
Since $f$ is Baire measurable, $\mu(f)$ is well defined by Lemma~\ref{L:integr-for-baire}. Set $x=r(\mu)$. Assuming $\mu(f)\notin \Gamma(x)$, $\Gamma(x)$, as a convex closed set, can by separated from $\mu(f)$ by some element from $F^*$, i.e.,
there exists $\tau\in F^*$ and $c\in\er$ such that
\[
(\Re \tau)(\mu(f))>c>\sup (\Re \tau)(\Gamma(x)).
\]
Let
\[
\varphi(y)=\sup (\Re \tau)(\Gamma(y)),\quad y\in X.
\]
Since the graph of $\Gamma$ is convex, $\varphi$ is a  concave function on $X$. Moreover, if $\Gamma$ is upper semicontinuous, then $\varphi$ is upper semicontinuous; and if $\Gamma$ is lower semicontinuous, then $\varphi$ is lower semicontinuous.
 In both case we get by  \cite[Proposition 4.7]{lmns}
\[
%\begin{aligned}
c<(\Re \tau)(\mu(f))=\int_X (\Re \tau)(f(y))\di \mu(y)\le \int_X \varphi(y)\di\mu(y)\le \varphi(x)<c.
%\end{aligned}
\]
This contradiction finishes the proof.
\end{proof}

Now we can give the proof:

\begin{proof}[Proof of Theorem~\ref{T:selekceusc}.]
(S) By \cite[Theorem]{jaro82c} (which is an improved version of \cite[Theorem 2]{jaro82}) and Lemma~\ref{L:baire}(c) there is a selection $f$ from $\Gamma$ which belongs to $\C_1(X,F)$. Set $g(x)=\delta_x(f)$ for $x\in X$.
By Theorem~\ref{T:dilation} and Lemma~\ref{L:fra-0}(iii) we have $g\in\fra_2(X,F)$ and by Lemma~\ref{L:sel-convex}, $g$ is a selection of $\Gamma$.

(R) We proceed in the same way. We find a selection $f$ in $\C_1(X,F)$ and set $f_1=\odd f$. Observe that $f_1$ is still a selection from $\Gamma$. Indeed, if $x\in X$, then $(x,f(x))$ and $(-x,f(-x))$ belong to the graph of $\Gamma$. Since the graph of $\Gamma$ is symmetric, $(x,-f(-x))$ belongs there, too. By convexity we conclude that $(x,f_1(x))=(x,\odd f(x))$ belongs there as well. Further, set $g=Tf_1$. By Theorem~\ref{T:dilation} and Lemma~\ref{L:fra-0}(iii) we have $g\in\fra_2(X,F)$. Moreover, let $x\in X$ be arbitrary. Fix a maximal representing measure $\mu$ of $x$. Then
$$g(x)=T(x)(f_1)=\odd\mu(f_1)= \mu(f_1)\in \Gamma(x)$$
by Lemmata~\ref{L:sel-convex} and~\ref{l:odd-miry}(d), hence $g$ is a selection from $\Gamma$. It remains to observe that $g$ is odd since $T$ is odd.

(C) This case is analogous to the case (R). We find a selection $f$ in $\C_1(X,F)$ and set $f_1=\hom f$. Observe that $f_1$ is still a selection of $\Gamma$. Indeed, if $x\in X$ and $t\in[0,2\pi]$, then $(e^{it}x,f(e^{it}x))$ belongs to the graph of $\Gamma$. Since the graph is homogeneous, it contains also $(x,e^{-it}f(e^{it}x))$, hence $e^{-it}f(e^{it}x)\in\Gamma(x)$. Since $\Gamma(x)$ is closed and convex, we get $f_1(x)=\hom f(x)\in\Gamma(x)$. The rest of the proof is the same as in case (R).
\end{proof}

It is natural to ask whether we can obtain a nice affine selection even in case $X$ is not metrizable. We stress that the Jayne-Rogers theorem requires metrizability both of the domain space and of the range space. However, if $F$ is separable, it easily follows from the Kuratowski--Ryll-Nardzewski selection theorem (see \cite{kurn} or \cite[Theorem 5.2.1]{srivastava}) that $F$ admits a Borel measurable selection, in fact a selection of the first Borel class in the sense of \cite{spurny-amh}. However, an affine Borel selection need not exist even in the scalar case. This is illustrated by the following example.
In the same example we show that the class of the affine selection in the metrizable case cannot be improved.

\begin{example}\label{ex:selekce}
There are simplices $X_1$, $X_2$ and upper semicontinuous mappings $\Gamma_i\colon X_i\to\er$ with closed values, bounded range and convex graph for $i=1,2$ such that the following assertions hold:
\begin{itemize}
	\item[(i)] $X_1$ is metrizable and $\Gamma_1$ admits no affine Baire-one selection.
	\item[(ii)] $X_2$ is non-metrizable and $\Gamma_2$ admits no affine Borel selection.
\end{itemize}
\end{example}

\begin{proof} Let $A\subset [0,1]$ be any subset. Let $K$, $\A$ and $X$ be defined as in Example~\ref{ex:dikobraz}.
We again consider $K$ canonically embedded into $X$. Fix an arbitrary set $B\subset A$ and define two real-valued functions on $K$ by
$u_0=\chi_{B\times\{1\}}$ and $v_0=\chi_{(B\times\{1\})\cup ([0,1]\times\{0\})}$. Then $u_0\le v_0$. We further define two functions on $X$ by the formulas
$$\begin{aligned}
u(x)&=\sup\{ h(x)\setsep h\in\fra(X,\er), h\le u_0\mbox{ on }K\}, \\
v(x)&=\inf\{ h(x)\setsep h\in\fra(X,\er), h\ge v_0\mbox{ on }K\}.
\end{aligned}$$
Then $u$ is a lower semicontinuous convex function on $X$ and $v$ is an upper semicontinuous concave function on $X$. Moreover,
clearly $u\le v$ (if $h_1,h_2\in\fra(X,\er)$ with $h_1\le u_0$ and $h_2\ge v_0$ on $K$, then $h_1\le h_2$ on $\ext X$, hence $h_1\le h_2$ on $X$). Therefore the formula $\Gamma(x)=[u(x),v(x)]$, $x\in X$, defines a bounded upper semicontinuous mapping with nonempty closed values and convex graph.

We claim that $u\r_K=u_0$ and $v\r_K=v_0$. This follows from abstract results \cite[Propositions 3.48 and 3.55]{lmns}, but it can be seen also directly:

Since $0\le u_0$, it follows that $u(x)=0=u_0(x)$ for $x\in K\setminus(B\times \{1\})$. If $x=(b,1)$ for some $b\in B$, consider
the function $f_b=\chi_{(b,1)}-\chi_{(b,-1)}$. Then $f_b\in \A$ and $f_b\le u_0$. Since $f_b$ defines a function in $\fra(X,\er)$, we conclude that $u(x)=1=u_0(x)$.

Since $v_0\le 1$, we get $v(x)=1=v_0(x)$ for $x\in (B\times\{1\})\cup ([0,1]\times\{0\})$. If $x=(a,-1)$ for some $a\in A$, then $1+f_a\ge v_0$, thus
$v(x)=0=v_0(x)$. If $x=(a,1)$ for some $a\in A\setminus B$, then $1-f_a\ge v_0$, thus $v(x)=0=v_0(x)$.

In particular, $u=v$ on $A\times\{-1,1\}$. Let $g$ be any affine selection from $\Gamma$. Then
$$f(a,0)=\begin{cases} \frac12, & a\in B,\\ 0, & a\in A\setminus B.\end{cases}$$

Therefore, to construct $X_1$ it is enough to take $A$ to be a countable dense subset of $[0,1]$ and $B\subset A$ such that both $B$ and $A\setminus B$ are dense. To construct $X_2$ it is enough to take $A$ to be an uncountable Borel set and $B\subset A$ a non-Borel subset.
\end{proof}

\section{Sharpness of results and open problems}

In this final section we collect several open questions and discuss which of the results are sharp.
We start by pointing out a gap between Theorem~\ref{T:bap} and Example~\ref{E:cap}.

\begin{question} Let $X$ be a compact convex set and $E$ a Banach space having the approximation property
(or the compact approximation property). Does any affine function $f\in\C_1(X,E)$ belong to $\fra_1(X,E)$?
\end{question}

The next question concerns optimality of constants in Theorem~\ref{T:bap}.

\begin{question} Let $X$ be a compact convex set and $E$ a Banach space with the bounded approximation property.
Does any affine function $f\in\C_1(X,B_E)$ belong to $\fra_1(X,B_E)$? 
\end{question}

We continue with problems concerning Theorem~\ref{T:dilation}. In case of a  metrizable simplex we have proved that the mapping $T$ is of class $\fra_1$ when the target space is $\M^1(X)$. The same proof gives a slightly stronger conclusion that $T$ is of class $\fra_1$ when the target space is $\M^1(\ov{\ext X})$ and similar statement in cases (R) and (C). However the following question seems to be open.

\begin{question} Let $X$ be a metrizable simplex. Is $T\in\fra_1(X,\M^1(\ext X))$? Is the analogous statement valid for $L_1$-preduals?
\end{question}

As we have remarked above, the metrizability assumption was used in the proof of Theorem~\ref{T:dilation} in an essential way.
Further, the metrizability assumption cannot be omitted or just weakened to the Lindel\"of property of $\ext X$ by Example~\ref{ex:dikobraz}. However, the general question when the operator $T$ belongs to the class $\C_1$ or even $\fra_1$ seems to be hard. A sufficient condition is metrizability of $X$ or closedness of $\ext X$. But these conditions are not necessary. It is easy to construct a simplex $X$ when $T$ is of class $\fra_1$ even though $X$ is not metrizable and $\ext X$ is not closed. For example, the simplex constructed by the method used in \cite[Theorem 4]{kalenda-bpms} starting with
$K=[0,\omega_1]$ and $A=\{\omega\}$. (This construction is based upon the well-known constructions from \cite{stacey} and \cite{bi-de-le}.) It seems that the following question is natural. (By an $(S,R,C)$ compact set we mean a compact convex set which is either a simplex of the dual unit ball of a real or complex $L_1$-predual.)

\begin{question}\label{q:kan} Let $X$ be an (S,R,C) compact set. Is $T$ of the class $\C_1$ or even $\fra_1$ if one of the below conditions is satisfied?
\begin{itemize}
	\item $\ext X$ is $K$-analytic.
	\item $\ext X$ is a Baire set.
	\item $\ext X$ is $F_\sigma$.
\end{itemize}
\end{question}

Let us further remark that the operator $T$ cannot be `too non-measurable' since it is strongly affine due to Lemma~\ref{l:te-sa}. More precisely, it follows from the proof of this lemma that in cases (S) and (R) $Tf$ is a uniform limit of a sequence of differences of lower semicontinuous functions for any $f\in\C(X,\er)$.
In case (C) the same property is enjoyed by $\Re Tf$ and $\Im Tf$ for each $f\in\C(X,\ce)$ by the proof of \cite[Lemma 4.12]{petracek-spurny}. Hence, in all three cases, $Tf$ is Borel measurable for any scalar continuous function $f$. (In fact, in this case $Tf$ is of the first Borel class in the sense of \cite{spurny-amh}.) It easily follows that $T^{-1}(U)$ is a Borel set (of the first additive class in the sense of \cite{spurny-amh}) for any cozero set $U$ in the respective set of measures (but it need not be Borel measurable as the spaces from Example~\ref{ex:dikobraz} easily show).
However, the function $Tf$ need not be of the first Baire class -- an example is given in \cite[Theorem 4]{kalenda-bpms}.
In fact, in the quoted example $T$ is not in the class $\fra_\alpha$ for any $\alpha<\omega_1$ since $Tf$ need not even be a Baire function (it is a consequence of the quoted example and \cite[Theorem 5]{kalenda-bpms}).

We continue by looking at Theorem~\ref{T:aff-baire}.  It is sharp even in the scalar case, i.e. the shift $\alpha\to1+\alpha$ may occur. In the case (S) it was showed in \cite{spurnytams} for $\alpha=2$ and in \cite{spu-zel} for general $\alpha\in[2,\omega)$. The same examples can be easily transferred to the cases (R) and (C). What we further do not know is the following question.

\begin{question} Let $X$ be an (S,R,C) compact set such that $\ext X$ is a Lindel\"of resolvable set. Can one replace $1+\alpha$ by $\alpha$ in Theorem~\ref{T:aff-baire}?
\end{question}

Recall that a subset of a $X$ is said to be \emph{resolvable} (or a \emph{H-set}) if its characteristic function has the point of continuity property. For more details see \cite{spu-ka}. Within metrizable spaces this class is defined in \cite[\S12, III]{kuratowski} using an equivalent condition. Further, a subset of a completely metrizable space is resolvable if and only if it is simultaneously $F_\sigma$ and $G_\delta$. Thus, the answer is positive if $X$ is metrizable.
Further, the answer is also positive in the scalar case. The case (S) follows from \cite{spu-ka}, the case (R) from \cite[Theorem 1.4]{lusp}, the case (C) from \cite[Theorem 2.23]{lusp-complex}.

The conclusion of Theorem~\ref{T:dirichlet} is optimal since it follows from the optimality of Theorem~\ref{T:aff-baire} (the above-mentioned examples from \cite{spurnytams} and  \cite{spu-zel} are metrizable). However, it is not clear whether it is optimal for $\alpha=0$. 

\begin{question} Let $X$ be an (S,R,C) compact with $\ext X$ Lindel\"of, $F$ be a \fr space and $f\colon\ext X\to F$ a bounded continuous (odd, homogeneous) map. Can $f$ be extended to an element of $\fra_1(X,F)$?
\end{question}

In the scalar case the answer is positive (by \cite{Jel,lusp23,lusp-complex}), but in the vector case our proof gives just an extension in $\fra_2(X,F)$. Moreover, if $g$ is such an extension, the scalar version yields $\tau\circ g\in\fra_1(X,\ef)$ for each $\tau\in F^*$.

Another question is the following.

\begin{question} Let $X$ be an (S,R,C) compact set such that $\ext X$ is a Lindel\"of resolvable set. Can one replace $1+\alpha$ by $\alpha$ in Theorem~\ref{T:dirichlet}?
\end{question}

In the scalar case the answer is positive (by \cite{spu-ka,lusp23,lusp-complex}). In the vector case we are able to prove it
only if $\ext X$ is an $F_\sigma$ set, hence the answer is positive if $X$ is metrizable.

Theorem~\ref{T:dirichlet} deals with extending Baire mappings. It is natural to ask whether we can extend Baire measurable mappings. If the target space $F$ is separable, then any Baire measurable mapping with values in $F$ which is defined on a regular Lindel\"of space is a Baire mapping of the respective class (by Lemma \ref{L:baire-n}).
Hence the question of extending Baire measurable mappings reduces to the question whether a Baire measurable mapping has separable range.

\begin{question} Let $X$ be a compact convex set with $\ext X$ Lindel\"of. Let $f:\ext X\to F$ be a Baire measurable mapping with values in a \fr space. Is the range of $f$ separable?\end{question}

The answer is positive if $\ext X$ is a Baire set (or, more generally, a $K$-analytic set) -- see Lemma~\ref{L:baire}(a).
For a general Lindel\"of (even separable metric) space in place of $\ext X$ the answer is negative by \cite[Example 2.4(3)]{koumou}. But we do not know whether an example can be of the form of  $\ext X$.

Finally, one can ask whether the assumption that $\ext X$ is Lindel\"of is necessary. It cannot be just omitted
(even in the scalar case, see, e.g., \cite[Theorem 4]{kalenda-bpms}). However, the following question asked already in \cite{Jel} is still open.

\begin{question}\label{q:lind} Let $X$ be a simplex such that each bounded continuous real-valued function on $\ext X$ can be extended to an affine Baire function on $X$. Is necessarily $\ext X$ Lindel\"of?
Does the analogous implication hold for dual balls of $L_1$-preduals?
\end{question}

Let us remark that the answer is positive within a special class of `Stacey simplices' by \cite[Theorems 2 and 6]{kalenda-bpms},
but the general case is not clear.

Finally, Theorem~\ref{T:weakDP} is clearly sharp and Theorem~\ref{T:selekceusc} is sharp due to Example~\ref{ex:selekce}.

\section*{Acknowledgement}

We are grateful to the referee for valuable comments which helped to improve the presentation
of the paper. In particular, we thank for a simplification of the proof of Lemma~\ref{selekce5}.

%\bibliography{l1predual-vector}\bibliographystyle{siam}

\begin{thebibliography}{10}

\bibitem{alfsen}
{\sc E.~Alfsen}, {\em Compact convex sets and boundary integrals},
  Springer-Verlag, New York, 1971.
\newblock Ergebnisse der Mathematik und ihrer Grenzgebiete, Band 57.

\bibitem{bednar1972concerning}
{\sc J.~Bednar and H.~Lacey}, {\em Concerning {B}anach spaces whose duals are
  abstract {L}-spaces}, Pacific Journal of Mathematics, 41 (1972), pp.~13--24.

\bibitem{bi-de-le}
{\sc E.~Bishop and K.~de~Leeuw}, {\em The representations of linear functionals
  by measures on sets of extreme points}, Ann. Inst. Fourier. Grenoble, 9
  (1959), pp.~305--331.

\bibitem{capon}
{\sc M.~Capon}, {\em Sur les fonctions qui v\'erifient le calcul
  barycentrique}, Proc. London Math. Soc. (3), 32 (1976), pp.~163--180.

\bibitem{casazza}
{\sc P.~G. Casazza}, {\em Approximation properties}, in Handbook of the
  geometry of {B}anach spaces, {V}ol. {I}, North-Holland, Amsterdam, 2001,
  pp.~271--316.

\bibitem{castillo2009extending}
{\sc J.~Castillo and J.~Su{\'a}rez}, {\em Extending operators into
  {L}indenstrauss spaces}, Israel Journal of Mathematics, 169 (2009),
  pp.~1--27.

\bibitem{kranti}
{\sc M.~K. Das}, {\em A note on complex {$L_1$}-predual spaces}, Internat. J.
  Math. Math. Sci., 13 (1990), pp.~379--382.

\bibitem{dean}
{\sc D.~W. Dean}, {\em The equation {$L(E,\,X^{\ast\ast})=L(E,\,X)^{\ast\ast}$}
  and the principle of local reflexivity}, Proc. Amer. Math. Soc., 40 (1973),
  pp.~146--148.

\bibitem{dieckmann1994korovkin}
{\sc G.~Dieckmann}, {\em Korovkin theory in {L}indenstrauss spaces.},
  Mathematica Scandinavica, 75 (1994), pp.~240--254.

\bibitem{dulin}
{\sc Y.~Duan and B.-L. Lin}, {\em Characterizations of {$L^1$}-predual spaces
  by centerable subsets}, Comment. Math. Univ. Carolin., 48 (2007),
  pp.~239--243.

\bibitem{effros-real}
{\sc E.~Effros}, {\em On a class of real {B}anach spaces}, Israel Journal of
  Mathematics, 9 (1971), pp.~430--458.

\bibitem{effros}
\leavevmode\vrule height 2pt depth -1.6pt width 23pt, {\em On a class of
  complex {B}anach spaces}, Illinois J. Math., 18 (1974), pp.~48--59.

\bibitem{engelking}
{\sc R.~Engelking}, {\em General topology}, PWN---Polish Scientific Publishers,
  Warsaw, 1977.
\newblock Translated from the Polish by the author, Monografie Matematyczne,
  Tom 60. [Mathematical Monographs, Vol. 60].

\bibitem{hhz}
{\sc M.~Fabian, P.~Habala, P.~H{\'a}jek, V.~Montesinos, and V.~Zizler}, {\em
  Banach space theory}, CMS Books in Mathematics/Ouvrages de Math\'ematiques de
  la SMC, Springer, New York, 2011.
\newblock The basis for linear and nonlinear analysis.

\bibitem{fakhoury1}
{\sc H.~Fakhoury}, {\em Pr{\'e}duaux de {L}-espace: Notion de centre}, Journal
  of Functional Analysis, 9 (1972), pp.~189--207.

\bibitem{fakhoury2}
\leavevmode\vrule height 2pt depth -1.6pt width 23pt, {\em Une characterisation
  des {L}-espaces duax}, Bull. Sci. Math., 96 (1972), pp.~129--144.

\bibitem{fonf1978massiveness}
{\sc V.~Fonf}, {\em Massiveness of the set of extreme points of the dual ball
  of a {B}anach space. {P}olyhedral spaces}, Functional Analysis and its
  Applications, 12 (1978), pp.~237--239.

\bibitem{fonf}
{\sc V.~P. Fonf, J.~Lindenstrauss, and R.~R. Phelps}, {\em Infinite dimensional
  convexity}, in Handbook of the geometry of {B}anach spaces, {V}ol. {I},
  North-Holland, Amsterdam, 2001, pp.~599--670.

\bibitem{frolik-bulams}
{\sc Z.~Frol{\'{\i}}k}, {\em A measurable map with analytic domain and
  metrizable range is quotient.}, Bull. Amer. Math. Soc., 76 (1970),
  pp.~1112--1117.

\bibitem{frolikcmj}
\leavevmode\vrule height 2pt depth -1.6pt width 23pt, {\em A survey of
  separable descriptive theory of sets and spaces}, Czechoslovak Math. J., 20
  (95) (1970), pp.~406--467.

\bibitem{gasparis2002contractively}
{\sc I.~Gasparis}, {\em On contractively complemented subspaces of separable
  {$L_1$}-preduals}, Israel Journal of Mathematics, 128 (2002), pp.~77--92.

\bibitem{jarchow}
{\sc H.~Jarchow}, {\em Locally convex spaces}, B. G. Teubner, Stuttgart, 1981.
\newblock Mathematische Leitf{\"a}den. [Mathematical Textbooks].

\bibitem{jaro82}
{\sc J.~E. Jayne and C.~A. Rogers}, {\em Upper semicontinuous set-valued
  functions}, Acta Math., 149 (1982), pp.~87--125.

\bibitem{jaro82c}
\leavevmode\vrule height 2pt depth -1.6pt width 23pt, {\em Correction to:
  ``{U}pper semicontinuous set-valued functions'' [{A}cta {M}ath. {\bf 149}
  (1982), no. 1-2, 87--125; {MR}0674168 (84b:54038)]}, Acta Math., 155 (1985),
  pp.~149--152.

\bibitem{Jel}
{\sc F.~Jellett}, {\em On affine extensions of continuous functions defined on
  the extreme boundary of a {C}hoquet simplex}, Quart. J. Math. Oxford Ser.
  (1), 36 (1985), pp.~71--73.

\bibitem{affperf}
{\sc O.~Kalenda and J.~Spurn\'y}, {\em Preserving affine {B}aire classes by
  perfect affine maps}, Quaest. Math.
\newblock DOI:10.2989/16073606.2015.1073813, preprint available at
  http://arxiv.org/abs/1501.05118.

\bibitem{kalenda-bpms}
{\sc O.~F.~K. Kalenda}, {\em Remark on the abstract {D}irichlet problem for
  {B}aire-one functions}, Bull. Pol. Acad. Sci. Math., 53 (2005), pp.~55--73.

\bibitem{kalenda-spurny}
{\sc O.~F.~K. Kalenda and J.~Spurn{\'y}}, {\em Extending {B}aire-one functions
  on topological spaces}, Topology Appl., 149 (2005), pp.~195--216.

\bibitem{koumou}
{\sc G.~Koumoullis}, {\em A generalization of functions of the first class},
  Topology Appl., 50 (1993), pp.~217--239.

\bibitem{kuratowski}
{\sc K.~Kuratowski}, {\em Topology. {V}ol. {I}}, New edition, revised and
  augmented. Translated from the French by J. Jaworowski, Academic Press, New
  York, 1966.

\bibitem{kurn}
{\sc K.~Kuratowski and C.~Ryll-Nardzewski}, {\em A general theorem on
  selectors}, Bull. Acad. Polon. Sci. S\'er. Sci. Math. Astronom. Phys., 13
  (1965), pp.~397--403.

\bibitem{lacey}
{\sc H.~Lacey}, {\em The isometric theory of classical {B}anach spaces},
  Springer-Verlag, New York, 1974.
\newblock Die Grundlehren der mathematischen Wissenschaften, Band 208.

\bibitem{Lau1973}
{\sc K.-S. Lau}, {\em The dual ball of a {L}indenstrauss space.}, Mathematica
  Scandinavica, 33 (1973), pp.~323--337.

\bibitem{lazar}
{\sc A.~Lazar}, {\em The unit ball in conjugate {$L_1$} spaces}, Duke
  Mathematical Journal, 39 (1972), pp.~1--8.

\bibitem{lazar-sel}
{\sc A.~J. Lazar}, {\em Spaces of affine continuous functions on simplexes},
  Trans. Amer. Math. Soc., 134 (1968), pp.~503--525.

\bibitem{lali}
{\sc A.~J. Lazar and J.~Lindenstrauss}, {\em Banach spaces whose duals are
  {$L_{1}$} spaces and their representing matrices}, Acta Math., 126 (1971),
  pp.~165--193.

\bibitem{limaroy}
{\sc {\AA}.~Lima and A.~K. Roy}, {\em Characterizations of complex
  {$L^1$}-preduals}, Quart. J. Math. Oxford Ser. (2), 35 (1984), pp.~439--453.

\bibitem{lusp-complex}
{\sc P.~Ludv\'{\i}k and J.~Spurn\'y}, {\em Baire classes of complex {$L_1$}
  preduals}, Czech. Math. J., \rm(to appear), preprint available at
  http://www.karlin.mff.cuni.cz/$\sim$spurny/articles.php.

\bibitem{lusp2}
{\sc P.~Ludv{\'{\i}}k and J.~Spurn{\'y}}, {\em Baire classes of
  {$L_1$}-preduals and {$C^*$}-algebras}, Illinois J. Math., \rm (to appear),
  preprint available at
  http://www.karlin.mff.cuni.cz/$\sim$spurny/articles.php.

\bibitem{lusp}
\leavevmode\vrule height 2pt depth -1.6pt width 23pt, {\em Descriptive
  properties of elements of biduals of {B}anach spaces}, Studia Math., 209
  (2012), pp.~71--99.

\bibitem{lusp23}
\leavevmode\vrule height 2pt depth -1.6pt width 23pt, {\em Baire classes of
  nonseparable {$L_1$} preduals}, Quart. J. Math., 66 (2015), pp.~251--263.

\bibitem{lmnss03}
{\sc J.~Luke{\v{s}}, J.~Mal{\'y}, I.~Netuka, M.~Smr{\v{c}}ka, and
  J.~Spurn{\'y}}, {\em On approximation of affine {B}aire-one functions},
  Israel J. Math., 134 (2003), pp.~255--287.

\bibitem{lmns}
{\sc J.~Luke{\v{s}}, J.~Mal{\'y}, I.~Netuka, and J.~Spurn{\'y}}, {\em Integral
  representation theory}, vol.~35 of de Gruyter Studies in Mathematics, Walter
  de Gruyter \& Co., Berlin, 2010.
\newblock Applications to convexity, Banach spaces and potential theory.

\bibitem{lusky1977separable}
{\sc W.~Lusky}, {\em On separable {L}indenstrauss spaces}, Journal of
  Functional Analysis, 26 (1977), pp.~103--120.

\bibitem{lusky-comp}
\leavevmode\vrule height 2pt depth -1.6pt width 23pt, {\em Every
  {$L_1$}-predual is complemented in a simplex space}, Israel Journal of
  Mathematics, 64 (1988), pp.~169--178.

\bibitem{lusky-compl}
\leavevmode\vrule height 2pt depth -1.6pt width 23pt, {\em Every separable
  {$L_1$}-predual is complemented in a {$C^*$}-algebra}, Studia Math., 160
  (2004), pp.~103--116.

\bibitem{MeSta}
{\sc S.~Mercourakis and E.~Stamati}, {\em Compactness in the first {B}aire
  class and {B}aire-1 operators}, Serdica Math. J., 28 (2002), pp.~1--36.

\bibitem{nielsenolsen}
{\sc N.~J. Nielsen and G.~H. Olsen}, {\em Complex preduals of {$L_{1}$} and
  subspaces of {$l^{n}_{\infty }(C)$}}, Math. Scand., 40 (1977), pp.~271--287.

\bibitem{odro}
{\sc E.~Odell and H.~P. Rosenthal}, {\em A double-dual characterization of
  separable {B}anach spaces containing {$l^{1}$}}, Israel J. Math., 20 (1975),
  pp.~375--384.

\bibitem{olsen-sel}
{\sc G.~H. Olsen}, {\em Edwards' separation theorem for complex {L}indenstrauss
  spaces with application to selection and embedding theorems}, Math. Scand.,
  38 (1976), pp.~97--105.

\bibitem{petracek-spurny}
{\sc P.~Petr\'a\v{c}ek and J.~Spurn\'y}, {\em A characterization of complex{
  $L_1$}-preduals via a complex barycentric mapping}.

\bibitem{phelps-choquet}
{\sc R.~R. Phelps}, {\em Lectures on {C}hoquet's theorem}, vol.~1757 of Lecture
  Notes in Mathematics, Springer-Verlag, Berlin, second~ed., 2001.

\bibitem{rao82}
{\sc T.~S. S. R.~K. Rao}, {\em Characterizations of some classes of
  {$L^{1}$}-preduals by the {A}lfsen-{E}ffros structure topology}, Israel J.
  Math., 42 (1982), pp.~20--32.

\bibitem{rao85}
\leavevmode\vrule height 2pt depth -1.6pt width 23pt, {\em A note on real
  structure in complex {$L^1$}-preduals}, Arch. Math. (Basel), 45 (1985),
  pp.~267--269.

\bibitem{rogalski}
{\sc M.~Rogalski}, {\em Op\'erateurs de {L}ion, projecteurs bo\'eliens et
  simplexes analytiques}, J. Functional Analysis, 2 (1968), pp.~458--488.

\bibitem{Ross-Stone}
{\sc K.~A. Ross and A.~H. Stone}, {\em Products of separable spaces}, Amer.
  Math. Monthly, 71 (1964), pp.~398--403.

\bibitem{roy1979convex}
{\sc A.~Roy}, {\em Convex functions on the dual ball of a complex
  {L}indenstrauss space}, Journal of the London Mathematical Society, 2 (1979),
  pp.~529--540.

\bibitem{sray-om}
{\sc J.~Saint-Raymond}, {\em Fonctions convexes sur un convexe born\'e
  complet}, Bull. Sci. Math. (2), 102 (1978), pp.~331--336.

\bibitem{sray}
\leavevmode\vrule height 2pt depth -1.6pt width 23pt, {\em Fonctions convexes
  de premi\`ere classe}, Math. Scand., 54 (1984), pp.~121--129.

\bibitem{spurny-cejm}
{\sc J.~Spurn{\'y}}, {\em The {D}irichlet problem for {B}aire-one functions},
  Cent. Eur. J. Math., 2 (2004), pp.~260--271 (electronic).

\bibitem{spurny-aus}
{\sc J.~Spurn{\'y}}, {\em Affine {B}aire-one functions on {C}hoquet simplexes},
  Bull. Austral. Math. Soc., 71 (2005), pp.~235--258.

\bibitem{spurny-wdp}
{\sc J.~Spurn{\'y}}, {\em The weak {D}irichlet problem for {B}aire functions},
  Proc. Amer. Math. Soc., 134 (2006), pp.~3153--3157 (electronic).

\bibitem{spu-aus}
{\sc J.~Spurn{\'y}}, {\em The {D}irichlet problem for {B}aire-two functions on
  simplices}, Bull. Aust. Math. Soc., 79 (2009), pp.~285--297.

\bibitem{spurnytams}
\leavevmode\vrule height 2pt depth -1.6pt width 23pt, {\em Baire classes of
  {B}anach spaces and strongly affine functions}, Trans. Amer. Math. Soc., 362
  (2010), pp.~1659--1680.

\bibitem{spurny-amh}
\leavevmode\vrule height 2pt depth -1.6pt width 23pt, {\em Borel sets and
  functions in topological spaces}, Acta Math. Hungar., 129 (2010), pp.~47--69.

\bibitem{spu-ka}
{\sc J.~Spurn{\'y} and O.~Kalenda}, {\em A solution of the abstract {D}irichlet
  problem for {B}aire-one functions}, J. Funct. Anal., 232 (2006),
  pp.~259--294.

\bibitem{spu-zel}
{\sc J.~Spurn{\'y} and M.~Zelen{\'y}}, {\em Baire classes of strongly affine
  functions on simplices and on {$C^*$}-algebras}, J. Funct. Anal., 267 (2014),
  pp.~3975--3993.

\bibitem{srivastava}
{\sc S.~M. Srivastava}, {\em A course on {B}orel sets}, vol.~180 of Graduate
  Texts in Mathematics, Springer-Verlag, New York, 1998.

\bibitem{stacey}
{\sc P.~J. Stacey}, {\em Choquet simplices with prescribed extreme and \v
  {S}ilov boundaries}, Quart. J. Math. Oxford Ser. (2), 30 (1979),
  pp.~469--482.

\bibitem{talagrand}
{\sc M.~Talagrand}, {\em A new type of affine {B}orel function.}, Mathematica
  Scandinavica, 54 (1984), pp.~183--188.

\bibitem{thomas}
{\sc G.~E.~F. Thomas}, {\em Integration of functions with values in locally
  convex {S}uslin spaces}, Trans. Amer. Math. Soc., 212 (1975), pp.~61--81.

\bibitem{utter}
{\sc U.~Uttersrud}, {\em Geometrical properties of subclasses of complex
  {$L_1$}-preduals}, Israel J. Math., 72 (1990), pp.~353--371 (1991).

\bibitem{vesely}
{\sc L.~Vesel{\'y}}, {\em Characterization of {B}aire-one functions between
  topological spaces}, Acta Univ. Carolin. Math. Phys., 33 (1992),
  pp.~143--156.

\bibitem{wu76}
{\sc D.~Wulbert}, {\em An approximation proof in the theory of complex
  {$L_{1}$}-preduals}, in Approximation theory, {II} ({P}roc. {I}nternat.
  {S}ympos., {U}niv. {T}exas, {A}ustin, {T}ex., 1976), Academic Press, New
  York, 1976, pp.~575--580.

\bibitem{wu78}
{\sc D.~E. Wulbert}, {\em Real structure in complex {$L_{1}$}-preduals}, Trans.
  Amer. Math. Soc., 235 (1978), pp.~165--181.

\end{thebibliography}

\end{document}